\documentclass[a4paper,11pt]{amsart}

\usepackage[USenglish]{babel}
\usepackage[T1]{fontenc} %The package allows the user to select font encodings
\usepackage[utf8]{inputenc} %The package translates various standard and other input encodings into a ‘LaTeX internal language’.

\usepackage{amsfonts} %An extended set of fonts for use in mathematics
\usepackage{amsmath} %It adapts for use in LaTeX most of the mathematical features found in AMS-TeX
\usepackage{amsthm} %The package facilitates the kind of theorem setup typically needed in American Mathematical Society publications.
\usepackage{lmodern}

\usepackage[colorlinks]{hyperref}
\hypersetup{
  linkcolor=[RGB]{192,0,0},
  citecolor=[RGB]{0, 120, 0},
  urlcolor=[rgb]{0.6, 0.2, 0.2}
}
\usepackage[alphabetic]{amsrefs}

\usepackage{amssymb}

\usepackage[nameinlink]{cleveref}

\usepackage[dotinlabels]{titletoc}

\usepackage{scalerel}

\usepackage{xcolor}

\usepackage{fix-cm}

\usepackage{bm}

\usepackage{derivative}

\usepackage{fancyhdr}

\usepackage{array}
\usepackage{braket}
\usepackage{scalefnt}
\usepackage{lipsum}

\usepackage{mathtools}
\usepackage{thmtools}
\usepackage{graphicx}
\usepackage[normalem]{ulem} %The package provides an \ul (underline) command which will break over line ends
\usepackage{tabularx}
\usepackage{comment}

\usepackage{paralist}
\usepackage{tensor}

\usepackage[font=small,labelfont=bf,margin=1.5cm]{caption}

\usepackage{tikz}
\usepackage{tikz-cd}
\usepackage{tikz-3dplot}
\usetikzlibrary{cd}
\usetikzlibrary{arrows}
\usetikzlibrary{matrix}
\usetikzlibrary{positioning}
\usetikzlibrary{decorations.pathreplacing}
\usetikzlibrary{calc,tikzmark}

\usepackage{nicematrix}
\usepackage{xcolor}
\definecolor{red}{rgb}{1,0,0}
\definecolor{darkred}{RGB}{192,0,0}

\newcommand{\eps}{\varepsilon}

\newcommand{\ot}{\otimes}

\newcommand{\ootimes}{ \otimes \dots \otimes }

\DeclareMathOperator{\CenOp}{Cen}
\newcommand{\cen}[1]{\CenOp_{#1}}

\newcommand{\bfS}{\mathbf{S}}

\newcommand{\bfd}{\mathbf{d}}

\newcommand{\bfu}{\mathbf{u}}
\newcommand{\bfv}{\mathbf{v}}
\newcommand{\bfw}{\mathbf{w}}
\newcommand{\bfx}{\mathbf{x}}
\newcommand{\bfy}{\mathbf{y}}
\newcommand{\bfz}{\mathbf{z}}

\newcommand{\calB}{\mathcal{B}}

\newcommand{\calD}{\mathcal{D}}

\newcommand{\bbC}{\mathbb{C}}

\newcommand{\bbN}{\mathbb{N}}

\newcommand{\mfm}{\mathfrak{m}}

\newcommand{\bfalpha}{\boldsymbol{\alpha}}

\newcommand{\injects}{\hookrightarrow}

\renewcommand{\tilde}[1]{\widetilde{#1}}

\newcommand{\Id}{\mathrm{Id}}

\DeclareMathOperator{\Ann}{Ann}

\DeclareMathOperator{\cha}{char}

\DeclareMathOperator{\End}{End}

\DeclareMathOperator{\Ker}{Ker}
\DeclareMathOperator{\im}{Im}

\DeclarePairedDelimiter{\pa}{\langle}{\rangle}

\DeclarePairedDelimiter{\an}{\langle}{\rangle}

\newcommand{\GL}{\mathrm{GL}}

\DeclareMathOperator{\Gr}{Gr}

\newcommand{\hook}{\,\lrcorner}

\usepackage[left=1in, right=1in, top=1in, bottom=1in, includefoot, headheight=13.6pt]{geometry}

\usepackage{enumitem}
\usepackage{adjustbox}

\allowdisplaybreaks

\numberwithin{equation}{section}
\theoremstyle{definition}
\newtheorem{defn}[equation]{Definition}
\theoremstyle{plain}
\newtheorem{teo}[defn]{Theorem}

\newtheorem{problem}{Problem}

\newtheorem{prop}[defn]{Proposition}

\newtheorem{lem}[defn]{Lemma}

\newtheorem{cor}[defn]{Corollary}
\theoremstyle{remark}
\newtheorem{rem}[defn]{Remark}

\newtheorem{exam}[defn]{Example}
\newtheorem{example}[defn]{Example}
\allowdisplaybreaks

\counterwithin{figure}{section}

\usepackage[textwidth=0.8in]{todonotes}
\title{Detecting direct sums of tensors and their limits}
\date{\today}

\author{Stefano Canino}
\address[Stefano Canino]{Dipartimento di Matematica, Università degli Studi di Trento, Via Sommarive 14, Povo, Trento, 38123, Italy.\newline
Wydział Matematyki, Informatyki i Mechaniki, Uniwersytet Warszawski,
 ul.~Stefana Banacha
2, 02-097 Warsaw, Poland.}
\email{stefano.canino@unitn.it}

\author{Cosimo Flavi}
\address[Cosimo Flavi]{Wydział Matematyki, Informatyki i Mechaniki, Uniwersytet Warszawski,
 ul.~Stefana Banacha
2, 02-097 Warsaw, Poland.}
\email{c.flavi@uw.edu.pl}

\author{Joachim Jelisiejew}
\address[Joachim Jelisiejew]{Wydział Matematyki, Informatyki i Mechaniki, Uniwersytet Warszawski,
 ul.~Stefana Banacha
2, 02-097 Warsaw, Poland.}
\email{j.jelisiejew@uw.edu.pl}
\thanks{Supported by National Science Centre grant 2023/50/E/ST1/0033.}

\keywords{Centroids, direct sum, tensors, classifications.}
\subjclass[2020]{Primary 14N07; secondary 15A69, 68Q15}

\begin{document}
\newcommand{\jjtodo}[1]{[{\color{violet}$\spadesuit$ #1}]}
\newcommand{\jjadd}{[{\color{violet}\texttt{ref}}]}
\newcommand{\jjch}[1]{\textcolor{violet}{#1}}
\newcommand{\kk}{\Bbbk}%

\maketitle

\begin{abstract}
    We generalize Mammana's classification of limits of direct sums to more than two factors. We also extend it from polynomials to arbitrary Segre-Veronese format, generalising and unifying results of  Buczy{\'n}ska-Buczy{\'n}ski-Kleppe-Teitler, Hwang, Wang, and Wilson. Remarkably, in such much more general setup it is still possible to characterise the possible limits. Our proofs are direct and based on the theory of centroids, in particular avoiding the delicate Betti number arguments.
\end{abstract}

\section{Introduction}
\noindent
The results of this paper tie together three independent research directions:
\begin{enumerate}
    \item finding and characterising direct sums of tensors, with two or more summands;
    \item analysing the geometry of the fibers of the gradient map;
    \item {giving new explicit conditions for a tensor to be of minimal border rank.}
\end{enumerate}

Let $T$ be a tensor, that is, an element of the space $S^{d_1} V_1 \ot S^{d_2} V_2
\ot\cdots \ot S^{d_e} V_e$. For example, if $e = 1$, then $T \in S^d V$ is a
homogeneous polynomial of degree $d$. We refer to this case as the
\emph{Veronese} format. The Veronese format is much better
understood thanks to methods coming from resolutions and regularity, which are
not available in the general setup. %If 
In the case where \[
d_1 = d_2 =  \cdots = d_e =1,
\]
we have $T\in V_1\ot \cdots \ot V_e$, which is the \emph{Segre} format.

\subsection{Direct sums}
To understand the complexity of $T$, it is very advantageous to present it as
a \emph{direct sum} $T = T' + T''$, where $T'$ and $T''$ are in disjoint
variables. This idea appears classically (e.g.~\cite[\S14.2]{BurgisserBook}), it is central to Strassen's Conjecture (for which
counterexamples exist, but are %few~
in very large dimensional spaces \cite{Shitov, BFPSS}), it generalises additive
decompositions of polynomials~\cite{Bernardi_Oneto_Taufer__On_GADs, BMT}, it
is related to smoothability via the notion of
\emph{cleaving}~\cite{BBKT_direct_sums, cjn13}, it appears for Gorenstein
algebras as their connected sum~\cite{Ananthnarayan_Avramov_Moore},
it is useful when numerically computing the rank, it appears prominently
in topology and geometry~\cite{Nemethi2, Nemethi, Ueda_Yoshinaga, Wang,
Wang_err} etc.
Of course, it is even better to present $T$ as a direct sum $T = T' + T'' +
T'''+ \cdots $ with more terms.

However, surprisingly little is known about the existence of such
decompositions. A notable exception is the case of polynomials.
\begin{teo}[{{\cite{Mammana}, \cite[Theorem~1.7 and (1)]{BBKT_direct_sums}},
    see also~\cite[\S4.1]{Beorchia}}]\label{refintro:BBKT:thm}
    Let $F\in S^d V$ be concise, that is, depend on all
    variables (see \autoref{ref:concise:def} for precise statement). Then the following are equivalent
    \begin{enumerate}
        \item\label{it:BBKTcases} $F$ is a direct sum $F_1 + F_2$ or a limit of such,
        \item\label{it:BBKTchar} $F$ is a direct sum $F_1 + F_2$ or it has the form
            \[
                F(\bfx, \bfy, \bfz) = \sum_{i=1}^k x_i \frac{\partial
                H(\bfy)}{\partial y_i} + G(\bfy, \bfz),
            \]
            where $\bfx = (x_1, \ldots ,x_k)$, $\bfy = (y_1, \ldots ,y_k)$,
            $\bfz = (z_1, \ldots ,z_{n-2k})$ form a basis of $V$ and $G$ and
            $H$ are homogeneous of degree $d$,
        \item\label{it:BBKTchar:gen} $\Ann(F)$ has a minimal generator of degree $d$.
    \end{enumerate}
\end{teo}
The above is very satisfactory, since~\eqref{it:BBKTchar:gen} is easy to verify
for a given $F$, while~\eqref{it:BBKTchar} gives a full characterisation of
possible cases. The proof {in~\cite{BBKT_direct_sums}} depends crucially on being in the Veronese
format: it uses the duality of the
minimal free resolution of the apolar algebra $S^d V^{\vee}/\Ann(F)$ and
classification of ideals with certain Betti numbers
nonvanishing~\cite[Theorem~8.18]{eisenbud:syzygies}. To be precise, the
statement says that an ideal with an almost longest possible quadratic strand of
the minimal resolution (with $\beta_{n-1,n}\neq 0$) is contained in a scroll.
It particular, it does not adapt to more than two factors in the case $S^d V$ and completely
breaks for other formats of tensors. This leaves two natural problems.

\begin{problem}\label{problem:formats}
    Generalise \autoref{refintro:BBKT:thm} to arbitrary tensors.
\end{problem}
\begin{problem}\label{problem:morefactors}
    Generalise \autoref{refintro:BBKT:thm} for polynomials to direct sums of more than two factors.
\end{problem}

In this paper, we solve \autoref{problem:formats} completely. The following is
a direct generalisation of \autoref{refintro:BBKT:thm}, yet the proof is
completely different, as the Betti number method used in~\cite{BBKT_direct_sums}
is unavailable for the
general format, as we explained above.
\begin{teo}[\autoref{teo: epsilon}, \autoref{cor: simmetrici epsilon}, $n=2$,
    \autoref{ref:centroidFromAnn:teo}]\label{refintro:CFJ:twoFactors}
    Let $
    T\in S^{d_1} V_1 
    \ot\cdots \ot S^{d_e} V_e$ be a concise (\autoref{ref:concise:def}) tensor. Then the following are equivalent
    \begin{enumerate}
        \item\label{it:CFJcases} $T$ is a direct sum $T_1 + T_2$ or a limit of such,
        \item\label{it:CFJchar} $T$ is a direct sum $T_1 + T_2$ or it has the form
            \begin{equation}\label{eq:limitDirSumTwoTensor}
                T = \sum_{j=1}^e \sum_{i=1}^{k_j} x_{j,i} \frac{\partial
                H}{\partial y_{j,i}} + G,
            \end{equation}
            where for every $j=1, \ldots ,e$, the elements $\bfx_j = (x_{j,1},
            \ldots ,x_{j,k_j})$, $\bfy_j = (y_{j,1}, \ldots ,y_{j,k_j})$,
            $\bfz_j = (z_{1,j}, \ldots ,z_{j,n_j-2k_j})$ form a basis of $V_j$
            and $G, H\in S^{d_1} V_1 \ot S^{d_2} V_2 \ot\cdots \ot S^{d_e}
            V_e$ are such that $G$ involves on $\bfy_{\bullet}$ and
            $\bfz_{\bullet}$ variables, while $H$ involves only
            $\bfy_{\bullet}$ variables.
        \item $\Ann(T)$ has a minimal generator of degree $(d_1, \ldots ,d_e)$.
    \end{enumerate}
\end{teo}

\subsection{Refinements using centroids}
Before proceeding to our results for \autoref{problem:morefactors}, let us
discuss how to distinguish between the direct sum case and its limits in
\autoref{refintro:BBKT:thm}\eqref{it:BBKTcases} and
\autoref{refintro:CFJ:twoFactors}\eqref{it:CFJcases}. We give a simple and
computationally very effective criterion based on the theory of
\emph{centroids}. Centroids of bilinear maps were introduced
in~\cite{Myasnikov} and further developed for maps coming from groups
in~\cite{Wilson} and generalised to other Segre formats
in~\cite{Brooksbank_Maglione_Wilson}. Centroids were rediscovered under the name of 111-algebras
in~\cite{JLP24}.

For the Veronese case, it is not a priori clear whether the centroid preserves
the symmetry. We show that it does and we define the centroid for the
Segre-Veronese case. This will also appear in~\cite{Jelisiejew_Bedlewo}, but these notes are in
preparation.

Given $T\in S^{d_1} V_1 \ot S^{d_2} V_2
\ot\cdots \ot S^{d_e} V_e$ and $X_j\in \End(V_j)$ for some $1\leq j\leq e$, we
denote by \[X_j\circ_j T\in S^{d_1} V_1 \ootimes (V_j \ot S^{d_j-1} V_j)
\ootimes S^{d_e} V_e\] the tensor resulting from applying $X_j$
on appropriate coordinate, see \autoref{rmk:indexing} and \autoref{example:action}.
The \emph{centroid} of $T$ is the subspace
%\begin{align}\label{eq:centroids}
    %\cen{T} = \Big\{ (X_1, \ldots ,X_e)\in \End(V_1)\times \ldots \times
    %\End(V_e)\ &| X_1\circ_1 T = X_2\circ_2 T =  \ldots = X_e\circ_e T\\\notag
    %&\forall_{1\leq j\leq m} X_j \circ_j T \in S^{d_1} V_1 \ot S^{d_2} V_2
%\ot\ldots \ot S^{d_e} V_e\Big\}.
%\end{align}
\begin{equation}\label{eq:centroids}
    \cen{T}=\Set{(X_1, \ldots ,X_e)\in \End(V_1)\times \cdots \times
    \End(V_e)|\begin{aligned}
    &X_i\circ_i T = X_j\circ_j T\\
    &X_j \circ_j T \in S^{d_1} V_1 \ot \cdots \ot S^{d_e} V_e\\
&\forall\, i,j=1,\dots,e
    \end{aligned}
    }.
\end{equation}
If $e\neq 1$, then the latter condition is redundant.
The centroid enjoys the following properties:
\begin{enumerate}
    \item the subspace $\cen{T}$ is a subalgebra of $\End(V_1)\times \cdots
        \times \End(V_e)$ and it is commutative;
    \item the projection of $\cen{T}$ onto a factor $\End(V_j)$ is an
        injective homomorphism of algebras;
    \item the dimension of the centroid satisfies
        \begin{equation}\label{eq:centroidDim}
            \dim \cen{T} = 1 + \mbox{number of minimal generators of $\Ann(T)$
            of degree }(d_1, \ldots ,d_e).
        \end{equation}
\end{enumerate}
The centroid can be computed effectively from definition or, alternatively,
when knowing the ideal $\Ann(T)$, see \autoref{ref:centroidFromAnn:teo}.
\begin{prop}[\autoref{prop:initial_case}, \autoref{ref:centroidFromAnn:teo}]\label{refintro:CFJlocalNaive:prop}
    In the setting of \autoref{refintro:CFJ:twoFactors}, assume that there is
    a single generator of $\Ann(T)$ in degree $(d_1, \ldots ,d_e)$. Then,
    there are two mutually exclusive possibilities
    \begin{enumerate}
        \item the centroid $\cen{T}$ is isomorphic to $\kk\times \kk$ and $T$
            is a direct sum.
        \item the centroid $\cen{T}$ is isomorphic to
           $\kk[\eps]/(\eps^2)$ and $T$ is a limit of direct sums
            as in~\eqref{eq:limitDirSumTwoTensor}.
    \end{enumerate}
\end{prop}
%\jjtodo{The above is not yet proven explicitly (I can handle this).}

Once we have a non-scalar element $r\in \cen{T}$, deciding which possibility
holds is easy. We can project $r$ to, for example, $\End(V_1)$ and compute its
minimal polynomial $\chi_r$, which has to have degree two
by~\eqref{eq:centroidDim}. If $\chi_r$ has no multiple roots, then the first possibility
holds. If $\chi_r$ has a double root, then the second possibility holds.

The possibilities in \autoref{refintro:CFJlocalNaive:prop} generalise to the
case without any assumptions on $\Ann(T)$. Namely, we have the following
result. After completing this paper, we learned that for $T$ a bilinear map, a
large part of this result appeared in~\cite[\S6.4]{Wilson}.
\begin{teo}\label{refintro:directSum:thm}
    Let $T\in S^{d_1} V_1 \ot S^{d_2} V_2 \ot\cdots \ot S^{d_e} V_e$ be
    concise and let $\cen{T}$ be its centroid. Let $\mfm_1, \ldots ,\mfm_n$ be
    the maximal ideals of $\cen{T}$. Then we have a direct sum decomposition
    \[
        T = T_1 +  \cdots + T_n,
    \]
    where $T_i$ are nonzero and $\cen{T_i} = (\cen{T})_{\mfm_i}$. This is the most refined direct sum
    decomposition: any other direct sum decomposition of $T$ comes from
    reindexing or grouping together terms of $\{T_1, \ldots ,T_n\}$.
    In particular, $T$ is not a direct sum if and only if $\cen{T}$ is local.
\end{teo}

\autoref{refintro:directSum:thm} implies that the case of local $\cen{T}$ is
the most interesting one. Indeed, this is also the hardest part of the analysis.

\subsection{General case}
Let us now generalise \autoref{refintro:CFJlocalNaive:prop} to the situation
without any assumption on the generators of $\Ann(T)$. Suppose that
$\cen{T}\neq \kk$ and take a non-scalar element $r\in \cen{T}$. We thus tackle
the following problem.

\begin{problem}
    Classify pairs: a tensor $T$ and an element $r\in \cen{T}$.
\end{problem}

We solve it completely. We also prove that the resulting pairs $(T, r)$ are
limits of direct sums with $n$ summands, where $n$ is the degree of the
minimal polynomial of $r$ where, by degree of the minimal polynomial of $r$, we mean the degree of the minimal polynomial of any coordinate of $r$. Indeed, the degree is the same for any factor and depends only on the structure of $\cen{T}$. Note also that the case $n=2$ is Mammana's classification given in
\autoref{refintro:BBKT:thm}.

The element $r$ can have multiple eigenvalues, if so,
\autoref{refintro:directSum:thm} yields a direct sum (see
\autoref{refintro:CFJ:higherPoly} below for precise description). This means
that the most interesting case is the local situation, where $r$ is nilpotent.
This is the content of the following \autoref{refintro:CFJ:higherPolyLocal}
which is our main technical result. We formulate it here only in the Veronese
case to increase clarity. The Segre case is given in \autoref{teo: epsilon}. The Segre-Veronese case follows easily from these.

%Going back to the setup with an element $r\in \Ann(T)$.
%What happens if $\cen{T}$ turns out to have minimal polynomial of degree
%higher than two?  This actually implies that the tensor $T$ has even more
%structure.  We describe it complete in the next two results, which are split
%and formulated only in Veronese case to increase clarity.

\begin{prop}[\autoref{cor: simmetrici epsilon},
    \autoref{prop:direct_sum_limit}]\label{refintro:CFJ:higherPolyLocal}
    Let $F\in S^d V$ be concise (see \autoref{ref:concise:def}) and let $n\geq
    1$ be a positive integer. The following are equivalent:
    \begin{enumerate}
        \item the centroid $\cen{F}$ contains a subalgebra isomorphic to {$\kk[\eps]/(\eps^n)$};
        \item there is a decomposition \[
       % V = \bigoplus_{1\leq q\leq n,\, 0\leq
    %        r\leq q-1}V^{(q, r)}
    V = \bigoplus_{0\leq
            r< q\leq n}V^{(q, r)}
            \]
            of vector spaces and fixed isomorphisms
            \[
            V^{(q, q-1)}\to V^{(q,q-2)} \to \cdots \to V^{(q, 0)},
            \]
            such that
            there exist
            homogeneous polynomials
            $F_1, \ldots ,F_n$, %with 
            where 
            \[
            F_k\in S^d \biggl(\bigoplus_{k\leq q\leq
            n} V^{(q,q-1)}\biggr),
            \] 
            such that
            \begin{equation}\label{eq:summation}
                F = \sum_{k=1}^n \sum_{\nu_1 + 2\nu_2+ \cdots + (n-1)\nu_{n-1} = k-1}
                \frac{1}{\nu_1! \nu_2! \cdots \nu_{n-1}!}D_1^{\nu_1}
                D_2^{\nu_2}
                \cdots D_{n-1}^{\nu_{n-1}} \hook F_k,
            \end{equation}
            where $D_i$ is the differential operator that is induced by the
            map $V^{(q, q-1)}\to V^{(q, q-1-i)}$, see~\eqref{eq:opAction}; in particular,
%            and
%            $\binom{d}{e_1, \ldots ,e_{n-1}} = \frac{d!}{e_1! \ldots
%                e_{n-1}!(d-\sum_i e_{i})!}$ is the multinomial coefficient.
                %In particular, 
                for $k=1$, we
            obtain simply $F_k = F_1$.
    \end{enumerate}
    Moreover, if these hold then $F$ is a limit of direct sums of the form
    $F^{(1)} + \cdots + F^{(n)}$, that is, with $n$ summands.
\end{prop}

This yields the case without any assumptions on $r$ as follows.
\begin{teo}[\autoref{refintro:directSum:thm}, \autoref{cor: simmetrici
    epsilon}, \autoref{prop:direct_sum_limit}]\label{refintro:CFJ:higherPoly}
    Let $F\in S^d V$ be concise (see \autoref{ref:concise:def}) and $n\geq 1$. Then the following are equivalent:
    \begin{enumerate}
        \item the centroid $\cen{F}$ contains an element $r$ whose minimal
            polynomial has degree $n$;
        \item\label{it:CFJcasesHigher} $F$ is a direct sum $F_{n_1} +  \cdots + F_{n_s}$, for some
            $s\geq 1$, where $n_1
            + \cdots + n_s = n$ and every summand $F_{n_i}$ has the
            form~\eqref{eq:summation} {with $n$ replaced by $n_i$}. (When $n_1 = \cdots = n_s
            = 1$, this just means that $F$ is a direct sum $F_1+ \cdots +
            F_n$).
%        \item $\Ann(F)$ has at least $n-1$ minimal generators of degree $d$
%            and $\cen{F}$ contains an element $r$ whose minimal
%            polynomial has degree $n$.
    \end{enumerate}
    Moreover, if these hold then $F$ is a limit of direct sums of the form
    $F^{(1)} + \cdots + F^{(n)}$, that is, with $n$ summands.
    Observe that~\eqref{it:CFJcasesHigher} describes all the cases,
    there are no limits involved.
\end{teo}

\begin{example}
    Suppose that $n = 2$. The decomposition from
    \autoref{refintro:CFJ:higherPoly} reduces to the above. Namely, we have $V
    = V^{(1, 0)} \oplus V^{(2,0)} \oplus V^{(2, 1)}$. Fix bases $\bfz$ of
    $V^{(1, 0)}$, $\bfy = (y_1, \ldots ,y_k)$ of $V^{(2, 1)}$ and $\bfx =
    (x_1, \ldots ,x_k)$ of $V^{(2, 0)}$. There is
    only one operator {$$D_1 = \sum_{i=1}^{k} x_i\frac{\partial}{\partial y_i}$$.}
    Let $G := F_1$ and $H := F_2$.
    The expression~\eqref{eq:summation} becomes
    \[
        F_1 + D_1 \circ F_2 = G(\bfy, \bfz) + \sum_{i=1}^k
        x_i\frac{\partial H(\bfy)}{\partial y_i}
    \]
    which agrees with \autoref{refintro:BBKT:thm}.
\end{example}

\begin{example}
    Take $n=3$. We have $6$ direct summands, which can be arranged in the
    diagram
    \[
        \begin{tikzcd}
            V^{(3, 0)} \\
            V^{(3, 1)}\ar[u, "\simeq"] & V^{(2, 0)}\\
            V^{(3, 2)}\ar[u, "\simeq"] & V^{(2, 1)}\ar[u, "\simeq"] & V^{(1, 0)}\\
        \end{tikzcd}
    \]
    Let $\bfz$, $\bfy$, $\bfx$, $\bfw$, $\bfv$, $\bfu$ be bases of $V^{(1, 0)}$, $V^{(2, 1)}$,
    $V^{(2, 0)}$, $V^{(3, 2)}$, $V^{(3, 1)}$, $V^{(3, 0)}$, respectively, let
    $k = \dim V^{(2, 1)} = \dim V^{(2, 0)}$ and $l = \dim V^{(3, 2)} = \dim
    V^{(3,1)} = \dim V^{(3, 0)}$. The
    expression~\eqref{eq:summation} becomes
    \begin{equation}\label{eq:expansionThree}
        F_1(\bfw, \bfy, \bfz) + \sum_{i=1}^k x_i\frac{\partial F_2(\bfw,
        \bfy)}{\partial y_i} + \sum_{i=1}^{l} v_i\frac{\partial F_2(\bfw,
        \bfy)}{\partial w_i} + \sum_{i=1}^{l} u_i\frac{\partial
        F_3(\bfw)}{\partial w_i} + \frac{1}{2}\left( \sum_{i=1}^l v_i\frac{\partial}{\partial w_i}
        \right)^2\circ F_3(\bfw)
    \end{equation}
    To get a concrete feeling for this expression, suppose now
    that \[
    V^{(2, 1)} = 0,\quad V^{(1, 0)} = 0,\quad V^{(3, 2)} = \kk w.
    \]
    Equation~\eqref{eq:expansionThree} becomes ${3}(w^2u + wv^2)$, which is
    isomorphic to multiplication tensor in $\kk[\eps]/(\eps^3)$.
\end{example}
\autoref{refintro:CFJ:higherPoly} solves \autoref{problem:morefactors} under
an additional assumption. As we explain in \cref{ssec:secants}, this
innocently looking problem is unsolvable in general, because it amounts to
characterising the whole boundary of the secant variety.

This solution to \autoref{problem:morefactors} generalises to other formats of
tensors, the Segre case is described in \autoref{teo: epsilon}.

\subsection{Fibers of the gradient map}

    \newcommand{\concisePolys}{\mathcal{C}onc_d}
    Consider %homogeneous polynomials 
    forms in $S^d V = \bbC[x_1,\ldots,x_n]_d$. To
    such a polynomial $F$, its \emph{gradient map} assigns the linear subspace
    \[
        \nabla F \coloneqq \left\langle \frac{\partial F}{\partial x_1},
        \ldots, \frac{\partial F}{\partial x_n}\right\rangle \subseteq
        S^{d-1} V
    \]
    This subspace is $n$-dimensional precisely when $F$ is
    concise (see \autoref{ref:concise:def}). Let $\concisePolys \subseteq
    S^d V$ denote the open subset consisting of concise polynomials. The gradient yields a morphism
    \[
        \nabla\colon \concisePolys\to \Gr(n, S^{d-1}V)
    \]
    which is generically one-to-one. The geometry of the gradient map is
    tightly connected to geometry of the hypersurface $F$,
    see~\cite{Ueda_Yoshinaga, Wang, Wang_err, Hwang}. It is also connected to
    the geometry of the Hessian map, see for
    example~\cite{Ciliberto_Ottaviani}.

    Concise polynomials $F$ for which
    $\nabla^{-1}(\nabla(F))$ is not a point, have been recently described by
    Hwang~\cite{Hwang}. He seems to be unaware
    of~\cite{Mammana, BBKT_direct_sums}, but his result perfectly fits into the picture
    and we can formulate it as follows.
    \begin{teo}[{\cite[Theorems~1.3-1.4]{Hwang} see also \cite{Mammana}}]\label{refintro:Hwang:thm}
        A concise polynomial $F$ has positive-dimensional
        $\nabla^{-1}(\nabla(F))$ if and only if it satisfies the
        condition \autoref{refintro:BBKT:thm}\eqref{it:BBKTchar}.
    \end{teo}
    In this context, direct sums are frequently called \emph{Thom-Sebastiani}
    polynomials.

    One could look for polynomials which yield higher-dimensional fibers.
    Our \autoref{refintro:CFJ:higherPoly} yields plenty of examples of such
    polynomials. The ones described in \autoref{refintro:CFJ:higherPolyLocal}
    are (generically) not of Thom-Sebastiani type.

    \begin{cor}
        The dimension of $\nabla^{-1}(\nabla(F))$ is equal to the number of
        minimal generators of $\Ann(F)$ of degree $\deg(F)$. In particular,
        for a polynomial $F$ satisfying the conditions of
        \autoref{refintro:CFJ:higherPoly}, we have $\dim
        \nabla^{-1}(\nabla(F))\geq n$.
    \end{cor}

%    The
%    following is well-known.
%    \begin{prop}
%        The dimension of $\nabla^{-1}(\nabla(F))$ is equal to the number of
%        minimal generators of $F$ of degree $\deg(F)$.
%    \end{prop}
%    By~\eqref{eq:centroidDim} this is also one less than the dimension of
%    $\cen{F}$.

%    Our
%    generalisation of \autoref{refintro:BBKT:thm}.
%    \begin{prop}
%    \end{prop}

    A natural generalisation of the gradient map is the (total)
    \emph{contraction} map, which maps a tensor $T\in S^{d_1} V_1 \ot S^{d_2} V_2
    \ot\cdots \ot S^{d_e} V_e$ to a tuple of spaces
    \begin{align}
        T(V_1^{\vee}) &\subseteq S^{d_1-1} V_1 \ot S^{d_2} V_2
    \ot\cdots \ot S^{d_e} V_e,\\\notag
        T(V_2^{\vee}) &\subseteq S^{d_1} V_1 \ot S^{d_2-1} V_2
    \ot\cdots \ot S^{d_e} V_e,\\\notag
     \qquad\vdots \\\notag
        T(V_e^{\vee}) &\subseteq S^{d_1} V_1 \ot S^{d_2} V_2
    \ot\cdots \ot S^{d_e-1} V_e.\notag
\end{align}
Again, this map is generically one-to-one (although we do not know of an
explicit reference for this fact). Geometrically, here we are
investigating hyperfsurfaces in products of projective spaces, which is of
interest in projective geometry. Our solution to \autoref{problem:formats} applies here as
well, yielding a source, but not a full classification, of tensors with large
fiber of the contraction map, see \autoref{teo: epsilon} together with
\autoref{ref:centroidFromAnn:teo}.

\subsection{Applications to secant varieties}\label{ssec:secants}

    Suppose now that $V_1, \ldots ,V_e$ all have the same dimension, equal to
    $m$ and $T\in S^{d_1}V_1 \ootimes S^{d_e}V_e$ is concise. It is a
    classical and very hard problem to decide whether $T$ lies in the $m$-th
    secant variety
    \[
        \sigma_m\left( \mathbb{P}V_1\times \cdots \times \mathbb{P}V_e \right)
        \subseteq \mathbb{P}\left( S^{d_1} V_1 \ot S^{d_2} V_2
        \ot\cdots \ot S^{d_e} V_e \right)
    \]
    of a Segre-Veronese embedding.
    This problem was one of the principal motivations
    of~\cite{Buczynska_Buczynski__border, JLP24}, who work in the Segre
    format with $e=3$.  In~\cite{Buczynska_Buczynski__border} Buczy{\'n}ska
    and Buczy{\'n}ski deduced that it is necessary to satisfy the
    111-condition. In~\cite{JLP24} the authors recasted this into a condition
    $\dim_{\kk} \cen{T} \geq m$. So it is known that
    \[
        \sigma_m\cap \left\{ \mbox{concise} \right\}\subseteq \left\{ [T]\ |\
            \mbox{concise, }\dim_{\kk} \cen{T}\geq m
        \right\}.
    \]
    The reverse inclusion is very false in general, this has connections to
    smoothability and cactus phenomena. We make a small step in the positive
    direction, as follows:
    \begin{prop}[\autoref{prop:direct_sum_limit}]
        Suppose that $\dim_{\kk}\cen{T}\geq m$ is generated by a single element
        $r\in \cen{T}$. Then $T$ lies in $\sigma_m$.
    \end{prop}
    Of course, if $T$ is a multiplication tensor in an algebra, then the above
    is clear and follows from irreducibility of the Hilbert scheme of
    $\mathbb{A}^1$. The main point of the Proposition is that it works for all
    tensors $T$.

%\subsubsection{Examples of wild tensors}

%Fix an integer $r$ and consider the $r$-th secant variety
%\[
%    \sigma_r\left( \mathbb{P}V_1\times \ldots \times \mathbb{P}V_e \right)
%    \subseteq \mathbb{P}\left( S^{d_1} V_1 \ot S^{d_2} V_2
%\ot\ldots \ot S^{d_e} V_e \right)
%\]
%of a Segre-Veronese embedding. These varieties, especially in the case of
%``pure'' Segre or Veronese embeddings, are investigated \jjadd{}, in
%particular, using apolarity, cactus schemes and generalised additive
%decompositions. It is known that all these methods apply to the ``tame''
%locus, where the so-called smothable and border ranks coincide.

%Very little is known about the wild locus \jjadd{}, especially in the
%non-polynomial setting. In this work, we give new examples of wild polynomials
%and tensors of minimal border rank.
%\autoref{problem:morefactors} for $(\dim V)$-factors is asking for a
%characterisation of polynomials which are limits of direct sums of the form
%$F_1 +  \ldots + F_{\dim V}$. By dimensional reasons, one has $F_i = \ell_i^d$
%for some linear forms $\ell_1, \ldots ,\ell_{\dim V}$. Giving such

%\jjtodo{To be continued}

\subsection*{Acknowledgements}
The authors would like to thank Giorgio Ottaviani for the beautiful lecture at \emph{Gianfranco
Casnati's Legacy} meeting, where we learned about Mammana's
work. The first author has been partially founded by the Italian Ministry of University and Research in the framework of the Call for Proposals for scrolling of final rankings of the PRIN 2022 call - Protocol no.~2022NBN7TL. The first and the second authors have been partially supported by the project \textit{Thematic Research Programmes}, Action I.1.5 of the program \textit{Excellence Initiative -- Research University} (IDUB) of the Polish Ministry of Science and Higher Education. 

\section{Preliminary notions}
\newcommand{\into}{\injects}

\subsection{Partially symmetric tensors:
notation}\label{ssec:PartiallSymmetric}
We work over an algebraically closed field $\kk$. The symbol $\ot$ denotes
tensoring over $\kk$. We will consider tensors $T\in S^{d_1} V_1 \ot \cdots
\ot S^{d_e} V_e$, but when $e = 1$, we will drop the indices and refer to
$T\in S^d V$. We also assume that $d_1, \ldots ,d_e\geq 1$.

\begin{defn}\label{ref:concise:def}
    A tensor $T\in S^{d_1} V_1 \ot \cdots \ot S^{d_e} V_e$ is \emph{concise}
    if there are no subspaces $V_1'\subseteq V_1$, \ldots ,$V_e'\subseteq
    V_e$, at least one of them proper, such that $T\in S^{d_1} V_1' \ot \cdots
    \ot S^{d_e} V_e'$.
\end{defn}

When dealing with symmetric
powers $S^d$ we tacitly assume that $\kk$ has characteristic zero or strictly
larger than $d$. For example, considering $T\in V_1 \ot V_2\ot\cdots \ot V_e$,
then we do not assume anything about the characteristic, while considering
$T\in S^2V_1 \ot V_2 \cdots \ot V_e$, we assume that $\cha \kk \neq 2$ only.

Thanks to the characteristic assumption, we can view all formats as
partially-symmetric subsets of the Segre format. Precisely speaking, let
\begin{equation}\label{eq:symmetric}
    S^{d} V \into \underbrace{V\ot \cdots \ot V}_d
\end{equation}
be the usual embedding of symmetric tensors, which maps an element $v_1\cdots v_d$, for $v_1, \ldots ,v_d\in V$, to an element
\[
    \frac{1}{d!}\sum_{\sigma\in\bfS_m} v_{\sigma(1)}\ot \cdots \ot
    v_{\sigma(d)}.
\]
A composition of~\eqref{eq:symmetric} yields an embedding
\begin{equation}\label{eq:desymmetrization}
    S^{d_1} V_1 \ot \cdots \ot S^{d_e} V_e \into \underbrace{V_1\ot \cdots \ot
    V_1}_{d_1} \ot \underbrace{V_2\ot \cdots \ot V_2}_{d_2}\ot \cdots \ot \underbrace{V_e\ot \cdots \ot V_e}_{d_e}
\end{equation}
where the right hand side has in total $d_1 + d_2 + \cdots + d_e$ factors.

Let us briefly discuss the actions of matrices on these tensors.  In the Segre
format, the situation is clear: having a tensor $T\in V_1 \ot \cdots \ot V_e$
and a matrix $X_j\in\End(V_j)$ for some $1\leq j\leq e$, we apply $X_j$ onto
the $j$-th factor to obtain another tensor $X_j\circ_j T\in V_1\ot \cdots \ot
V_e$.

In the general format, $T\in S^{d_1}V_1\ot \cdots \ot S^{d_e}V_e$, we have two
possibilities for an action of $X_j\in \End(V_j)$, where $1\leq j\leq e$.  From the Lie algebra
perspective, the space $\End(V_j)$ is the Lie algebra of $\GL(V_j)$ so it acts
naturally on $S^{d_j}V_j$ by differential operators. Consequently, it also
acts on the space $S^{d_1}V_1\ot \cdots \ot
S^{d_e}V_e$. Let us denote this action by $X_j\hook_j T$.
Explicitly, we have
\begin{equation}\label{eq:opAction}
    X_j \hook_j T_1\ot \cdots \ot (v_1 \cdots
    v_{d_j})\ot  \cdots \ot T_{e} = T_1\ootimes \biggl(
    \sum_{i=1}^{d_j} v_1\cdots v_{i-1} X_j(v_i)
    \cdots v_{d_j}
    \biggr)\ot  \cdots \ot T_{e}.
\end{equation}
From the more direct perspective, we can consider the image of
$T$ inside the bigger space using~\eqref{eq:desymmetrization} and then sum up
the action of $X_j$ via $\circ$ on \emph{each} of the coordinates { \[d_{1}+ \cdots
+d_{j-1}+1,\quad \ldots, \quad d_{1}+ \cdots +d_{j-1}+d_j.
\]
} It follows
from~\eqref{eq:opAction}, that the two operations agree.

A word of warning is necessary. It is \emph{not} true that one can identify
$\hook_j$ with any of the $\circ_{d_1+ \cdots +d_{j-1}+1}$, even in the
case of polynomials, as the following remark and example show.

\begin{rem}\label{rmk:indexing}
Let $T\in S^{d_1} V_1 \ot \cdots \ot S^{d_e} V_e$ and let $X_j\in
\End(V_j)$ for some fixed $1\leq j\leq m$. We can then view $T$ inside
$S^{d_1} V_1 \ootimes (V_j\ot S^{d_{j}-1}V_{j})\ootimes S^{d_e} V_e$ and apply
the operator $X_j$ to $V_j$, obtaining another element {$$X_j\circ_j T\in
S^{d_1} V_1 \ootimes (V_j\ot S^{d_{j}-1}V_{j})\ootimes S^{d_e} V_e.$$} The
symmetrization~\eqref{eq:symmetric} of this element is $\frac{1}{d_j}(X_j\hook T)$.
\end{rem}

\begin{example}\label{example:action}
    Let $F\in S^d V$. Take a basis $V = \pa{x_1, \ldots ,x_n}$ and the dual
    basis $\alpha_1, \ldots ,\alpha_n$ of $V^{\vee}$. Let $X\in
    \End(V) = V\ot V^{\vee}$ be given by $X = [\lambda_{ij}]$, so that {$$X =
    \sum_{i=1}^n \sum_{j=1}^n \lambda_{ij}x_i \ot \alpha_j$$.} We have
    \[
        X \hook F = \sum_{i=1}^n\sum_{j=1}^n \lambda_{ij}x_i\cdot \frac{\partial F}{\partial x_j}
    \]
    However, we have
    \[
        X\circ_1 F = \sum_{i=1}^n\sum_{j=1}^n \lambda_{ij}x_i\ot \frac{\partial
        F}{\partial x_j},
    \]
    which need not to be symmetric: already if $d = 2$, $F = x_1^2 + x_2^2$, and $X
    = x_1\ot \alpha_2$, we have \[X\circ_1 F = x_1\ot x_2 \neq x_2\ot x_1 =
    X\circ_2 F.\]
\end{example}

\subsection{Apolarity}\label{ssec:apolarity}

    Consider the ring $\calD \coloneqq S^{\bullet}V_1^{\vee}\ootimes
    S^{\bullet}V_{e}^{\vee}$ with its natural $\mathbb{N}^{e}$-grading. The
    ring $\calD$ is isomorphic to $S^{\bullet}(V_{1}^{\vee}\oplus \cdots
    \oplus V_{e}^{\vee})$.

    This ring acts on every $S^{d_1}V_1 \ootimes
    S^{d_e}V_e$, as we describe below. Let {$$T = T_1\ootimes T_e\in S^{d_1}V_1 \ootimes
    S^{d_e}V_e.$$}
    Let $\alpha\in V_j^{\vee}$. We
    define $\alpha\hook T$ by the formula resembling~\eqref{eq:opAction}:
    \[
        \alpha \hook T_1\ootimes (v_1 \cdots
        v_{d_j})\ootimes T_{e} = T_1\ootimes\biggl(
        \sum_{i=1}^{d_j} v_1 \cdots v_{i-1} \alpha(v_i)
        v_{i+1} \cdots v_{d_j}
        \biggr)\ot  \cdots \ot T_{e},
    \]
    where $\alpha(v_i)\in \kk$ are scalars. The ring $\calD$ is a polynomial
    algebra with linear forms spanned by all elements $1\ootimes \alpha\ootimes 1$, where
    $1\leq j\leq m$ and $\alpha\in V_{j}^{\vee}$, hence the above yields an
    action of $\calD$ on $S^{d_1}V_1 \ootimes
    S^{d_e}V_e$, called the \emph{apolarity action}. (For experts: this is the
    partial-derivation action, not the contraction action.)

    \begin{defn}
        The \emph{apolar ideal} $\Ann(T)$ is $\Set{ r\in \calD |
            r\hook T
        = 0}$. It is homogeneous with respect to the
        $\mathbb{N}^{e}$-grading on $\calD$.
    \end{defn}
    In the literature, the apolarity is defined classically in the polynomial
    setting, see for example~\cite[Appendix~A]{iakanev},
    \cite{BBKT_direct_sums}
    and it is vastly generalised in recent years~\cite{Galazka,
    Buczynska_Buczynski__border}. We observe that the
    setup above can be viewed purely in the polynomial setting.
    Namely, the space $S^{d_1}V_1 \ootimes
    S^{d_e}V_e$ embeds into $S^{d_1+ \cdots +d_e} (V_1\oplus \cdots \oplus
    V_e)$. The action of $\calD$ on $S^{d_1}V_1 \ootimes
    S^{d_e}V_e$ and the action of $S^{\bullet}(V_{1}^{\vee}\oplus \cdots
    \oplus V_{e}^{\vee})$ on $S^{d_1+ \cdots +d_e} (V_1\oplus \cdots \oplus
    V_e)$ agree.

    \begin{example}
        In \emph{Macaulay2}, the setup above is easily created as follows. For
        concreteness, assume $e=2$, $V_1 = \an{v_{1,1}, v_{1,2}}$, $V_2 =
        \an{v_{2,1}, v_{2,2}}$ and $T = v_{1,1}\ot v_{2,1}v_{2,2} + v_{1,2}\ot
        v_{2,1}^2\in V_1\ot S^{2}V_2$.
\begin{verbatim}
D = QQ[v_(1,1), v_(1,2)] ** QQ[v_(2,1), v_(2,2)];
T = v_(1,1)*v_(2,1)*v_(2,2) + v_(1,2)*v_(2,1)^2;
AnnT = inverseSystem(T);
#select(flatten entries mingens AnnT, gen -> degree gen == degree T)
\end{verbatim}
    \end{example}

\subsection{Centroids}
\begin{defn}\label{ref:centroid:def}
$T\in V_1\otimes\cdots\otimes V_e$ be a concise tensor. A tuple
\[
    (X_1,\dots, X_e)\in \End(V_1)\times\cdots\times\End(V_e)
\]
is \textit{compatible with $T$} if $X_1\circ_1 T=\dots=X_e\circ_e T$.
The vector subspace $\cen{T}$ consisting of all $e$-tuples compatible with $T$
is called the \textit {centroid of $T$}. For $r\in \cen{T}$ we denote $r\circ
T$ any of the equal tensors $X_1\circ_1 T$, $X_2 \circ_2 T$, \ldots ,
$X_e\circ_e T$.
\end{defn}

The following fundamental result is is stated
in~\cite[page~46]{Brooksbank_Maglione_Wilson} and proven in the case $e=3$
in~\cite[\S4]{JLP24}, the proof generalises immediately.
\begin{teo}[\cite{Brooksbank_Maglione_Wilson}, \cite{JLP24}]\label{ref:centroidIsSubalgebra:thm}
    Let $e\geq 3$.
    The subspace $\cen{T} \subseteq \End(V_1)\times \cdots \times \End(V_e)$
    is a commutative unital subalgebra. The projection of $\cen{T}$ to each of
    the factors is injective. For $r\in \cen{T}$ we have $r\circ T = 0$ if and
    only if $r = 0$.
\end{teo}
For more about the theory of centroids, we refer to \cite{Myasnikov, Wilson,
Brooksbank_Maglione_Wilson, JLP24, Jelisiejew_Bedlewo}. For a tensor $T\in S^{d_1}V_1\ot \cdots \ot S^{d_e} V_e$ we now show
that the centroid defined in \autoref{ref:centroid:def} identifies with the
centroid defined in~\eqref{eq:centroids}.

%\begin{rem}
%Given a vector space $V=\bbC^m$, we denote by $T(V)$ its tensor algebra. Given
%$$T=\sum_{i_1,\dots,i_d} t_{i_1,\dots,i_d}v_{i_1}\otimes\cdots\otimes v_{i_d}$$ we denote by $\Sym(T)$ its symmetric part, i.e.
%$$
%T=\sum_{\sigma\in\bfS_m}\sum_{i_1,\dots,i_d} t_{i_1,\dots,i_d}\sigma( v_{i_1}\otimes\dots\otimes v_{i_d}).
%$$
%We denote the set of symmetric tensors by $\Sym V\coloneqq\Set{T\in T(V)| T=\Sym(T)}$. Given $T_1,T_2\in\Sym V$, we denote by $T_1\odot T_2\coloneqq\Sym(T_1\otimes T_2)$ their symmetric product, which makes $\Sym(V)$ a $\bbC$-algebra.
%Recall that,  once a basis $(x_1,\dots,x_m)$ of $V$ is fixed, there is an isomorphism of graded $\bbC$-algebras $$\calT\colon\bbC[x_1,\dots,x_n]\to \Sym(V)$$
%such that, for any monomial $x_1^{d_1}\cdots x_m^{d_m}\in \bbC[x_1,\dots,x_m]_d$ we have
%$$\calT(x_1^{d_1}\cdots x_m^{d_m})=\frac{1}{d!}\sum_{\sigma\in \mfS_d}\sigma\left(\bigotimes_{i=1}^m x_i^{\otimes d_i}\right).$$
%We denote by $\calP\colon \Sym V\to\bbC[x_1,\dots,x_n]$ the inverse of $\calT$. If $(\alpha_1,\dots,\alpha_m)$ is the dual basis of $(x_1,\dots,x_m)$, then, for any $T\in T(V)$ of order $d$, we have
%$$T(\alpha_i)=\frac1d\calT\left(\frac{\partial\calP(T)}{\partial x_i}\right)$$
%where $T(\alpha_i)$ is any flattening of the tensor.
%\end{rem}

\begin{lem}\label{lemma:simmetriciCompatibili}
    Let $d_1 +  \cdots + d_e\geq 3$, let $T\in S^{d_1}V_1\ot \cdots \ot
    S^{d_e} V_e$ be a concise tensor, and let $(X_1, \ldots ,X_{d_1 +  \cdots + d_e})$
    be an element of $\cen{T}$. Then $X_1 =  \cdots = X_{d_1}$, $X_{d_1+1}=
    \cdots = X_{d_1 + d_2}$ and so on, so that the action of the tuple
    coincides with the action of $(X_{d_1}, X_{d_1+d_2}, \ldots ,X_{d_1 + d_2
    +  \cdots + d_e})$ as described in~\eqref{eq:centroids} in the
    introduction. Conversely, a tuple $(Y_1, \ldots ,Y_e)\in
    \End(V_1)\times \cdots \times \End(V_e)$ yields an element $(Y_1, \ldots
    ,Y_1, Y_2, \ldots ,Y_2,  \ldots, Y_e,  \ldots , Y_e)$ of $\cen{T}$ if
    and only if the following two conditions hold:
    \begin{enumerate}
        \item $\frac{1}{d_1}Y_1\hook_1 T = \frac{1}{d_2}Y_2\hook_2 T =
            \cdots  = \frac{1}{d_e}Y_e\hook_e
            T$;
        \item $Y_1\circ_{d_1} T$,  \ldots, $Y_e\circ_{d_1+ \cdots +d_e} T$ all
            lie in $S^{d_1}V_1\ot \cdots \ot S^{d_e} V_e$.
    \end{enumerate}
\end{lem}

\begin{proof}
    Let $(12)\in \bfS_{d_1 \cdots +d_e}$ denote the transposition of first and
    second factors. Choose any index $k$ different from $1,2$. We have
    \[
        X_1\circ_2 T=(12)\bigl(X_1\circ_1 (12)T\bigr)=(12)(X_1\circ_1
        T)=(12)(X_k\circ_k T)=\bigl(X_k\circ_k (12)T\bigr)=X_k\circ_k T=X_2\circ_2 T,
    \]
    so that $(X_1 - X_2)\circ_2 T = 0$. Since $T$ is concise, as
    in~\cite{JLP24}*{Lemma 3.1} we conclude that $X_1 = X_2$. All the other
    equalities are proven in the same way.

    To prove the assertion about $(Y_1, \ldots ,Y_e)$, assume first that the
    two conditions hold. For
    every $1\leq k\leq d_1$, let $(d_1k)$ be the transposition of $d_1$ with
    $k$. By the second condition, the tensor $Y_1\circ_{d_1} T$ is symmetric,
    so
    \[
        Y_1\circ_{k} T = (d_1k)(Y_1\circ_{d_1} (d_1k)T) =
        (d_1k)(Y_1\circ_{d_1} T) = Y_1 \circ_{d_1} T,
    \]
    which implies that $Y_1 \circ_1 T =  \cdots = Y_1 \circ_{d_1} T$. Since
    all these are equal, they are also equal to \[\frac{1}{d_1}\biggl(
    \sum_{k=1}^{d_1} Y_1 \circ_k T \biggr) = \frac{1}{d_1}(Y_1\hook_1 T).\]
    Similar arguments hold for $Y_2, \ldots ,Y_e$. The first condition now
    implies that \[(Y_1, \ldots
    ,Y_1, Y_2, \ldots ,Y_2,  \ldots, Y_e,  \ldots , Y_e)\] is in the centroid.
    The converse is similar.
\end{proof}

Directly from the definition, for every $1\leq i\leq m$, we have a
homomorphism of algebras $\cen{T}\to \End(V_i)$, hence $V_i$ is a
$\cen{T}$-module.

\subsubsection{Centroids and apolar ideals}
The linear-algebraic algorithm for computing the centroid is quite
self-evident from~\eqref{eq:centroids} and it is
implemented in \emph{Macaulay2} for example in~\cite[Auxiliary
files]{Jagiella_Jelisiejew}. Here we discuss an alternative
algorithm which allows to easily compute the dimension of $\cen{T}$ and, with
more effort, also the algebra structure.

\begin{lem}[nonsymmetric Euler formula]\label{ref:Euler:lem}
    Let $F\in S^{d}V$ be a symmetric polynomial. Let $V = \an{x_1, \ldots
    ,x_n}$. Then we have
    \[
        d\cdot F = \sum_{i=1}^n x_i \otimes \frac{\partial F}{\partial x_i},
    \]
    where the right-hand-side a priori lies in $V \ot S^{d-1}V$.
\end{lem}
\begin{proof}
    the space $S^d$ is spanned by $\{ \ell^d\ |\ \ell\in V\}$, so we can
    assume $F = \ell^d$. Write $\ell = \sum_{i=1}^n \lambda_i x_i$, then
    \[
        \sum_{i=1}^n x_i \otimes \frac{\partial F}{\partial x_i} =
        \sum_{i=1}^n x_i \otimes d\lambda_i \ell^{d-1} = \biggl(
        \sum_{i=1}^n \lambda_i x_i
        \biggr)\otimes d\cdot \ell^{d-1} = \ell \otimes d\cdot \ell^{d-1} =
        d\cdot \ell^d,
    \]
    as claimed.
\end{proof}

\begin{lem}\label{ref:totalSpan:lem}
    Let $T\in S^{d_1}V_1 \ootimes S^{d_e}V_e$ be concise. Let
    $\Ann(T)$ be its apolar ideal, as in~\autoref{ssec:apolarity}. Consider
    the space
    \begin{equation}\label{eq:totalSpan}
        \left\{ G\in S^{d_1}V_1 \ootimes S^{d_e}V_e\ |\ \forall_{1\leq j\leq
        e}\ V_j^{\vee}\hook_j G
        \subseteq V_j^{\vee}\hook_j T \right\}.
    \end{equation}
    The dimension of this space is one larger than the number of minimal
    generators of $\Ann(T)$ that have degree $(d_1, \ldots ,d_e)$.
\end{lem}
\begin{proof}
    Recall the polynomial ring $\calD = S^{\bullet}(V_{1}^{\vee}\oplus \cdots
    \oplus V_{e}^{\vee})$. Let $\bfd = (d_1, \ldots ,d_e)$. The catalecticant map
    \[
        \calD \to S^{\bullet}(V_1 \oplus  \cdots \oplus V_e)
    \]
    maps $r\in \calD$ to $r\hook T$, so it yields an isomorphism of vector
    spaces $\calD/\Ann(T)$ and $\calD\hook T$. The vector space
    $(\calD/\Ann(T))_{\bfd}$ is one-dimensional. Let $s$ be the number of
    minimal generators of $\Ann(T)$ in degree $\bfd$ and let $I\subseteq
    \Ann(T)$ be generated by all other minimal generators. Then
    $(\calD/I)_{\bfd}$ has dimension $1+s$ and
    $I_{\bfd'} = (\Ann(T))_{\bfd'}$ for every degree $\bfd' < \bfd$, in
    particular for $\bfd' = \bfd - (0,0, \ldots ,0,1,0, \ldots ,0)$, where $1$
    is on an arbitrary position.

    Consider the vector space $I^{\perp} \subseteq S^{\bullet}(V_1 \oplus
    \cdots \oplus V_e)$ defined as $\left\{ G\in S^{\bullet}(V_1 \oplus
    \cdots \oplus V_e)\ |\ I \hook G = 0 \right\}$. Since $I$ has no
    generators of degree $\bfd$, an element $G$ of
    degree $\bfd$ lies in
    this space if and only if $\alpha\hook G$ lies in $I^{\perp}$ for
    every $\alpha\in V_1^{\vee} \oplus  \cdots \oplus V_{e}^{\vee}$. But
    $\alpha\hook G$ is of degree smaller than $\bfd$, so $\alpha\hook
    G$ lies in $I^{\perp}$ if and only if $\Ann(T)\hook \alpha\hook G
    = 0$ if and only if $\alpha\hook G\in \calD\hook T$.
    It follows that the space~\eqref{eq:totalSpan} is equal to
    $I^{\perp}_{\bfd}$, hence has dimension $\dim_{\kk} (\calD/I)_{\bfd} =
    s+1$. This yields the claim.
\end{proof}

\begin{teo}[Centroid description for Segre-Veronese
        format]\label{ref:centroidFromAnn:teo}
        Let $T\in S^{d_1}V_1 \ootimes S^{d_e}V_e$ be concise. Let $\cen{T}$ be
        its centroid. Then the linear map $\cen{T} \to \cen{T}\circ T$ is
        bijective and the space $\cen{T} \circ T\subseteq S^{d_1}V_1
        \ootimes S^{d_e}V_e$ coincides with the space~\eqref{eq:totalSpan}. In
        particular, the dimension of the centroid $\cen{T}$ is given as in
        \autoref{ref:totalSpan:lem}.
\end{teo}

\begin{proof}
    It will be useful to employ both the definition of the centroid given in
    the introduction~\eqref{eq:centroids}.  It coincides with the one given in
    \autoref{ref:centroid:def} by \autoref{lemma:simmetriciCompatibili}.

    The map $\cen{T} \to \cen{T}\circ T$ is surjective by definition and
    injective by \autoref{ref:centroidIsSubalgebra:thm}, hence it is
    bijective. To prove that $\cen{T}\circ T$ coincides with the
    space~\eqref{eq:totalSpan}, we will show both containments.

    Take $r = (X_1, \ldots ,X_e)\in \cen{T}$ and let $G = r\circ T$. Fix an
    index $1\leq j\leq e$, and a basis $V_j = \an{x_1, \ldots ,x_n}$. Let $X_j
    = [\mu_{kl}]$ be the resulting matrix, so that $X_j(x_l) = \sum_{k}
    \mu_{kl}x_k$.
%    By~\eqref{eq:centroids}, we know that $G = X_j \circ_j
%    T$ is symmetric.
    By the nonsymmetric Euler's formula \autoref{ref:Euler:lem}
    we have
    \[
        {{d}}X_j \circ_j T = {{d}}\cdot G = \sum_{i=1}^n x_i \otimes (\alpha_i \hook_j G).
    \]
    Acting with $\alpha_i$ on both sides (for $i=1, \ldots ,n$), we obtain
    \[
        {{d}}(\alpha_iX_j)\circ_j T = \alpha_i \hook_j G,
    \]
    where $\alpha_iX_j\in V_j^{\vee}$, so that
    so $\alpha_i \hook_j G$ lies in
    $V_j^{\vee} \hook_j T$ for every $i=1, \ldots, n$. This shows that
    $\cen{T}\circ T$ is contained in~\eqref{eq:totalSpan}.
    Let $G$ lie in~\eqref{eq:totalSpan}. Again fix $1\leq j\leq e$ and a basis
    $V_j = \an{x_1, \ldots ,x_n}$. Write $\alpha_l\hook_j G =
    \sum_{k=1}^n \mu_{kl} \alpha_k\hook_j T$ for constants $\mu_{kl}\in
    \kk$. Let $X_j := [\mu_{kl}]$. By nonsymmetric Euler's formula again, we have
    \[
        {d\cdot} G = \sum_{l=1}^n x_l \ot (\alpha_l\hook_j G) = \sum_{l=1}^n x_l \ot
        \left( \sum_{k=1}^n \mu_{kl} \alpha_k \hook_j T \right) =
        \sum_{k,l=1}^n \mu_{kl} x_l \otimes \alpha_k \hook_j T = X_j \circ_j T,
    \]
    where $X_j = [\mu_{lk}]$, in particular ${d\cdot} G = X_j\circ_j T$ lies in
    $S^{d_1}V_1 \ootimes S^{d_e}V_e$. Arguing similarly for other coordinates, we
    obtain a tuple $(X_1, \ldots ,X_e)$ such that $X_j \circ_j T = G$ for
    every $j$. Thus $(X_1, \ldots ,X_e)$ lies in $\cen{T}$ and $G$ lies in
    $\cen{T}\circ T$.
\end{proof}

\begin{rem}
    The exact structure of $\cen{T}$ is easily extracted from the final part
    of the proof of \autoref{ref:centroidFromAnn:teo}. We refrain for
    formulating this explicitly, as the notation seems quite heavy.
\end{rem}

\subsection{Zero-dimensional algebra}
\begin{prop}[Chinese Remainder Theorem]\label{ref:CRT:thm}
    Let $R$ be a finite-dimensional $\kk$-algebra and let $\mfm_1, \ldots
    ,\mfm_n$ be all its maximal ideals. Then we have an isomorphism
    \begin{equation}\label{eq:CRT}
        R  \simeq R_{\mfm_1} \times \cdots \times R_{\mfm_n}.
    \end{equation}
\end{prop}
\begin{cor}[Idempotents in the Artinian case]\label{ref:idempotents:cor}
    Let $R$ be a finite-dimensional $\kk$-algebra presented as
    in~\eqref{eq:CRT}. Let $f\in R$. Then the following are equivalent
    \begin{enumerate}
        \item $f^2 = f$, that is, $f$ is an idempotent,
        \item on the right-hand-side of~\eqref{eq:CRT}, the element $f =
            (f_1, \ldots ,f_n)$ satisfies $f_i\in \{0, 1\}$ for every $i=1,
            \ldots ,n$.
    \end{enumerate}
\end{cor}
\begin{proof}
    An element $f = (f_1, \ldots ,f_n)$ satisfies $f^2 = f$ if and only $f_i^2
    = f_i\in R_{\mfm_i}$ for every $i=1, \ldots ,n$. Fix an index $i$. Since $f_i + (1-f_i)
    = 1$, it cannot happen that both $f_i$, $1-f_i$ belong to the maximal
    ideal of $R_{\mfm_i}$, hence one of them is invertible. If $f_i$ is
    invertible, then $f_i\cdot (1-f_i) = 0$ implies $1 - f_i = 0$, so $f_i =
    1$. Similarly, if $1-f_i$ is invertible, then $f_i = 0$.
\end{proof}

\section{Direct sums}

    In this section we consider direct sums and prove
    \autoref{refintro:directSum:thm}. Limits of direct sums will come in the
    next section.
    Before giving the proof, let us prove a special case. We keep the proof
    down-to-earth and we hope that it gives the reader the feeling for the
    general case.

\begin{prop}\label{prop:initial_case}
    Let $T\in S^{d_1}V_1\ootimes S^{d_e}V_e$ be a concise tensor. Then the following conditions are equivalent:
\begin{enumerate}
    \item\label{it:initialCaseLocal} the centroid of $T$ is not local,
    \item\label{it:initialCaseIdempotents} there exists a $\kk$-subalgebra of $\cen T$ is isomorphic to
        $\kk\times \kk$,
    \item\label{it:initialCaseIdemMatrix} there exists a non-zero element $(X_1,\dots,X_e)\in\cen T$ such that \[
        (X_1,\dots,X_e)\neq(\Id_{V_1},\dots,\Id_{V_e}),\qquad (X_1,\dots,X_e)^2=(X_1,\dots,X_e);\]
    \item\label{it:initialCaseDirectSum} there exist nonzero subspaces $V_{i,j}\subset V_i$, for $j=1,2$, such that
        $V_i=V_{i,1}\oplus V_{i,2}$,
        and there exist two concise tensors $T_j\in S^{d_1}V_{1,j}\ootimes
        S^{d_e}V_{e,j}$, for $j=1,2$, such that $T=T_1+ T_2$.
\end{enumerate}
\end{prop}
\begin{proof}
    To see the equivalence of~\eqref{it:initialCaseLocal}
    and~\eqref{it:initialCaseIdempotents} it is enough to use the Chinese Remainder
    Theorem, see \autoref{ref:CRT:thm}.

    To prove $\eqref{it:initialCaseIdempotents}\Rightarrow
    \eqref{it:initialCaseIdemMatrix}$, let $\varphi\colon \kk\times \kk
    \hookrightarrow \cen T$ an embedding of $\kk$-algebras and $(X_1,\dots,X_d)\coloneqq \varphi(0,1)$. Then
    \[
        (X_1,\dots,X_d)^2=\bigl(\varphi(0,1)\bigr)^2=\varphi\bigl((0,1)^2\bigr)=\varphi(0,1)=(X_1,\dots,X_d).
    \]
    Conversely, to obtain $\eqref{it:initialCaseIdemMatrix}\Rightarrow
    \eqref{it:initialCaseIdempotents}$, it is enough to note that the
    subalgebra of $\cen T$ generated by $(\Id_{V_1},\dots,\Id_{V_d})$ and
    $(X_1,\dots,X_d)$ is isomorphic to $\kk\times\kk$.

Let us prove $\eqref{it:initialCaseIdemMatrix}\Rightarrow\eqref{it:initialCaseDirectSum}$. Since
\[
(X_1^2,\dots,X_d^2)=(X_1,\dots,X_d)^2=(X_1,\dots,X_d),
\]
the endomorphisms $X_1,\dots,X_d$ are idempotent, that is, they are projections. Now, let us set
$Y_i\coloneqq\Id_{V_i}- X_i$,
and
$V_{i,1}\coloneqq\im X_i$, $V_{i,2}\coloneqq\im Y_i$, for every $i=1,\dots,d$.
Then, we have
$V_i=V_{i,1}\oplus V_{i,2}$.
Now, note that
\begin{gather*}
X_i\circ_i T\in V_1\otimes \cdots\otimes V_{i-1}\otimes V_{i,1}\otimes V_{i+1}\otimes \dots\otimes V_d,\\
Y_i\circ_i T\in V_1\otimes \cdots\otimes V_{i-1}\otimes V_{i,2}\otimes V_{i+1}\otimes \dots\otimes V_d.
\end{gather*}
Since $(X_1,\dots,X_d),(Y_1,\dots,Y_d)\in\cen T$, then
\[
T_1\coloneqq X_i\circ_i T\in V_{1,1}\ootimes V_{d,1},\qquad T_2\coloneqq Y_i\circ_i T\in V_{1,2}\ootimes V_{d,2}.
\]
Finally, we can write $T$ as
$$T=\Id_{V_i}\circ_i T=(X_i+Y_i)\circ_i T=T_1+T_2.$$
$\eqref{it:initialCaseDirectSum}\Rightarrow \eqref{it:initialCaseIdemMatrix}$. If we set $X_i$ the projection to $V_{i,1}$ for any $i=1,\dots,d$, then $(X_1,\cdots,X_d)\in\cen T$ and $(X_1,\cdots,X_d)^2=(X_1,\cdots,X_d)$.
\end{proof}

In the special case of $T$ fully symmetric, we obtain the following.
\begin{cor}
Let $F\in\kk[x_1,\cdots,x_m]_d$ a concise homogeneous polynomial of degree $d$. The following conditions are equivalent:
\begin{enumerate}
\item there exists a subalgebra of $\cen{F}$ which is isomorphic, as a
    $\kk$-algebra, to $\kk\times\kk$;
\item up to reordering the coordinates, there exists an integer $1\leq k\leq
    m$, such that $F=F_1+F_2$ where, $F_1\in\kk[x_1,\dots,x_k]_d$ and
    $F_2\in\kk[x_{k+1},\dots,x_m]_d$.
\end{enumerate}
\end{cor}

We stress that finding summands of direct sums in this way is very effective.
We give an easy explicit example to illustrate the method.
\begin{exam}
Let $V_1=V_2=V_3=\bbC^2$, $(a_1,a_2)$ the canonical basis of $V_1$, $(b_1,b_2)$ the canonical basis of $V_2$ and $(c_1,c_2)$ the canonical basis of $V_3$. Given
$$T=a_1\otimes b_1\otimes c_1+a_1\otimes b_2\otimes c_2+a_2\otimes b_1\otimes c_2+a_2\otimes b_2\otimes c_1\in V_1\otimes V_2\otimes V_3,$$
we want now to compute $\cen T$. In order to do that let us consider three generic endomorphisms $X\in\End(V_1),Y\in\End(V_2),Z\in\End(V_3)$ whose matrices with respect to the canonical bases are
$$X=\begin{pmatrix}x_{11} & x_{12}\\ x_{21} & x_{22}
\end{pmatrix}, \quad Y=\begin{pmatrix}y_{11} & y_{12}\\ y_{21} & y_{22}\end{pmatrix} \quad Z=\begin{pmatrix}z_{11} & z_{12}\\ z_{21} & z_{22}
\end{pmatrix}.$$
We have
\begin{align*}
X\circ_1 T&=(x_{11}a_1+x_{21}a_2)\otimes b_1\otimes c_1+(x_{11}a_1+x_{21}a_2)\otimes b_2\otimes c_2+(x_{12}a_1+x_{22}a_2)\otimes b_1\otimes c_2\\
&\hphantom{{}={}}+(x_{12}a_1+x_{22}a_2)\otimes b_2\otimes c_1\\
Y\circ_2 T&=a_1\otimes (y_{11}b_1+y_{21}b_2)\otimes c_1+a_1\otimes (y_{12}b_1+y_{22}b_2)\otimes c_2+a_2\otimes (y_{11}b_1+y_{21}b_2)\otimes c_2\\
&\hphantom{{}={}}+a_2\otimes (y_{12}b_1+y_{22}b_2)\otimes c_1\\
Z\circ_3 T&=a_1\otimes b_1\otimes (z_{11}c_1+z_{21}c_2)+a_1\otimes b_2\otimes (z_{12}c_1+z_{22}c_2)+a_2\otimes b_1\otimes (z_{12}c_1+z_{22}c_2)\\
&\hphantom{{}={}}+a_2\otimes b_2\otimes (z_{11}c_1+z_{21}c_2)
\end{align*}
By imposing $X\circ_1 T=Y\circ_2 T=Z\circ_3 T$ we get
$$\begin{cases}
x_{11}=y_{11}=z_{11},\\
x_{12}=y_{12}=z_{21},\\
x_{12}=y_{21}=z_{12},\\
x_{11}=y_{22}=z_{22},\\
x_{21}=y_{12}=z_{12},\\
x_{22}=y_{11}=z_{22},\\
x_{22}=y_{22}=z_{11},\\
x_{21}=y_{21}=z_{21},
\end{cases}\Longleftrightarrow\quad
\begin{cases}
x_{11}=y_{11}=z_{11}=x_{22}=y_{22}=z_{22},\\
x_{12}=y_{12}=z_{12}=x_{21}=y_{21}=z_{21}.
\end{cases}
$$
As a consequence, we get
\begin{align*}
\cen T&=\Set{(X,Y,Z)\in\End(V_1)\times\End(V_2)\times\End(V_3)| X=Y=Z=\begin{pmatrix} s & t \\ t & s
\end{pmatrix}\text{ for some }s,t\in\bbC}\\
&=\Biggl\langle\Biggl(\begin{pmatrix} 1 & 0 \\ 0 & 1
\end{pmatrix},\begin{pmatrix} 1 & 0 \\ 0 & 1
\end{pmatrix},\begin{pmatrix} 1 & 0 \\ 0 & 1
\end{pmatrix}\Biggr),\Biggl(\begin{pmatrix} 0 & 1 \\ 1 & 0
\end{pmatrix},\begin{pmatrix} 0 & 1 \\ 1 & 0
\end{pmatrix},\begin{pmatrix} 0 & 1 \\ 1 & 0
\end{pmatrix}\Biggr)\Biggr\rangle.
\end{align*}
Now, we consider the idempotent elements of $\cen T$
$$(X_1,X_2,X_3)=\dfrac{1}2\Biggl(\begin{pmatrix} 1 & 1 \\ 1 & 1
\end{pmatrix},\begin{pmatrix} 1 & 1 \\ 1 & 1
\end{pmatrix},\begin{pmatrix} 1 & 1 \\ 1 & 1
\end{pmatrix}\Biggr),$$
$$(Y_1,Y_2,Y_3)=(\Id_{V_1},\Id_{V_2},\Id_{V_3})-(X_1,X_2,X_3)=\dfrac{1}2\Biggl(\begin{pmatrix} 1 & -1 \\ -1 & 1
\end{pmatrix},\begin{pmatrix} 1 & -1 \\ -1 & 1
\end{pmatrix},\begin{pmatrix} 1 & -1 \\ -1 & 1
\end{pmatrix}\Biggr)$$
and we follow the proof of \Cref{prop:initial_case} to decompose the tensor $T$. We have
$$V_{1,1}=\langle a_1+a_2\rangle,\quad V_{1,2}=\langle a_1-a_2\rangle$$
and similarly for $B_{1,1},B_{1,2}$ and $C_{1,1},C_{1,2}$. In order to compute $T_1\in V_{1,1}\otimes B_{1,1}\otimes C_{1,1}$ and $T_2\in V_{1,2}\otimes B_{1,2}\otimes C_{1,2}$ such that $T=T_1+T_2$ it is enough to apply $X_1$ and $Y_1$ to $T$. In fact, we have
\begin{align*}
T_1=X_1\circ_1 T&=\frac12(a_1+a_2)\otimes( b_1\otimes c_1+b_2\otimes c_2+b_1\otimes c_2+b_2\otimes c_1)\\[1ex]
&=\frac12(a_1+a_2)\otimes(b_1+b_2)\otimes(c_1+c_2),\\[1ex]
T_2=Y_1\circ_1 T&=\frac12(a_1-a_2)\otimes( b_1\otimes c_1+b_2\otimes c_2-b_1\otimes c_2-b_2\otimes c_1)\\[1ex]
&= \frac12(a_1-a_2)\otimes(b_1-b_2)\otimes(c_1-c_2).
\end{align*}
In particular, we have
$$
T=\frac12\bigl((a_1+a_2)\otimes(b_1+b_2)\otimes(c_1+c_2)+(a_1-a_2)\otimes(b_1-b_2)\otimes(c_1-c_2)\bigr).
$$
\end{exam}

We are now ready to prove the general case of
\autoref{refintro:directSum:thm}.
    \begin{proof}[Proof of \autoref{refintro:directSum:thm}]
        First, let $R = \cen{T}$ and write a decomposition $R =
        R_{\mfm_1}\times \cdots \times R_{\mfm_n}$ as in \autoref{ref:CRT:thm}. Let
        $f_i := (0,0, \ldots ,0,1,0, \ldots ,0)\in R$, where $1$ is on the
        $i$-th coordinate.

        The elements $f_1, \ldots ,f_n\in \cen{T}$ are in particular acting on
        each $V_1, \ldots ,V_m$.  For every $j=1, \ldots ,m$, let $V_{j,i} :=
        f_i\cdot V_j \subset V_j$. Since $f_i\cdot f_j = 0$ for $i\neq j$ and
        $f_i^2 = f_i$, we obtain that
        \[
            V_j = V_{j,1}\oplus \cdots \oplus V_{j,n}
        \]
        for every $j=1, \ldots ,m$.

        Let $T_i := f_i\cdot T$. Then $T = 1\cdot T = (f_1+ \cdots +f_n) \cdot
        T = T_1 +  \cdots + T_n$. Moreover, for every $1\leq j\leq m$, the
        element $f_i\cdot T$ can be viewed as obtained by the action on the
        $j$-th coordinate, hence $T_i = f_i\cdot T\in V_1\ootimes V_{j-1}\ot
        V_{j,i}\ootimes V_e$. Intersecting over every $j$, we obtain that
        \[
            T_i \in V_{1,i}\ot V_{2,i}\ootimes V_{e, i}
        \]
        for every $i=1, \cdots ,n$. This yields the desired direct sum.
        Conciseness of $T$ implies the conciseness of $T_i$.
        To compute the centroid of $T_i$ for some $1\leq i\leq n$, observe
        that $(\cen{T})_{\mfm_i} =
        f_i\cen{T}$ as an algebra, so $(\cen{T})_{\mfm_i}$ is contained in
        $\cen{T_i}$. Conversely, if $r\in \cen{T_i}$, then $(0, \ldots 0,
        r,0, \ldots ,0)$ is in $\cen{T}$, hence $r\in (\cen{T})_{\mfm_i}$.

        To prove that every other direct sum comes from grouping together
        factors of this one, take a direct sum $T = T' + T''$.
        \autoref{prop:initial_case} implies that there is an element $f\in
        \cen{T}$ such that $f\cdot T = T'$ and $(1-f)\cdot T = T''$ and $f^2
        =f$. \autoref{ref:idempotents:cor} implies that $f = f_{i_1} + \cdots
        + f_{i_s}$ for some indices $1\leq i_1 < i_2< \cdots < i_s\leq e$. It
        follows that $T' = T_{i_1}+ \cdots +T_{i_s}$ and so the two-factor
        direct sum comes by grouping together the factors of $T = T_1 +
        \cdots + T_n$. For more factors, we obtain the same claim by induction.
    \end{proof}

%\begin{lem}
%Let $d\geq 3$, $T\in V^{\otimes d}$ be a concise symmetric tensor and $(X_1,\dots,X_d)\in\cen T$. Then, $X_1=\dots=X_d$ and $X_i\circ_i T$ is symmetric.
%\end{lem}
%\begin{proof}
%By \cite{JLP24}*{Lemma 3.1}, if we prove that $X_i\circ_j T=X_j\circ_j T$, that is $(X_i-X_j)\circ_i T=0$, we get $X_i=X_j$. Let $i\neq j\neq k\neq i$ and denote by $(ij)\in\mfS_n$ the transposition between $i$ and $j$. We have
%\[
%X_i\circ_j T=(ij)\bigl(X_i\circ_i (ij)T\bigr)=(ij)(X_i\circ_i T)=(ij)(X_k\circ_k T)=\bigl(X_k\circ_k (ij)T\bigr)=X_k\circ_k T=X_j\circ_j T,
%\]
%and this proves that $X_1=\dots=X_d$. A similar argument proves that $X_i\circ_i T$ is symmetric.
%\end{proof}
\section{The irreducible case}

\autoref{refintro:directSum:thm} allows to reduce to tensors which are not
direct sums, say \emph{irreducible tensors}. In this section we consider them.
We begin with the Segre case, which will immediately imply the general Segre-Veronese
case.
\begin{rem}\label{rem:Jordan}
Recall that, given a vector space $V$ and a nilpotent endomorphism $L\in\End(V)$ with nilpotency index $n$, it is possible to find a decomposition of $V$
{$$V=\bigoplus_{0\leq r<q\leq n}V^{(q,r)}$$}
such that
$$\Ker L^j=V^{(n,n-j)}\oplus V^{(n-1,n-j-1)}\oplus\dots\oplus V^{(j,0)}\oplus\Ker L^{j-1}$$
and $L(V^{(q,r)})=V^{(q,r+1)}$ for any $r=1,\dots,q-1$, so that the action is
as on \Cref{fig:Jordan}. Such a decomposition is called a \textit{Jordan decomposition} and a collection of bases $\bfx^{(q,r)}=(x^{(q,r)}_1,\dots,x^{(q,r)}_{t_q})$ of $V^{(q,r)}$ such that $L(x^{(q,r)}_s)=x^{(q,r+1)}_s$ for any $r=0,\dots,q-1$ is called a \textit{Jordan basis}.
\end{rem}
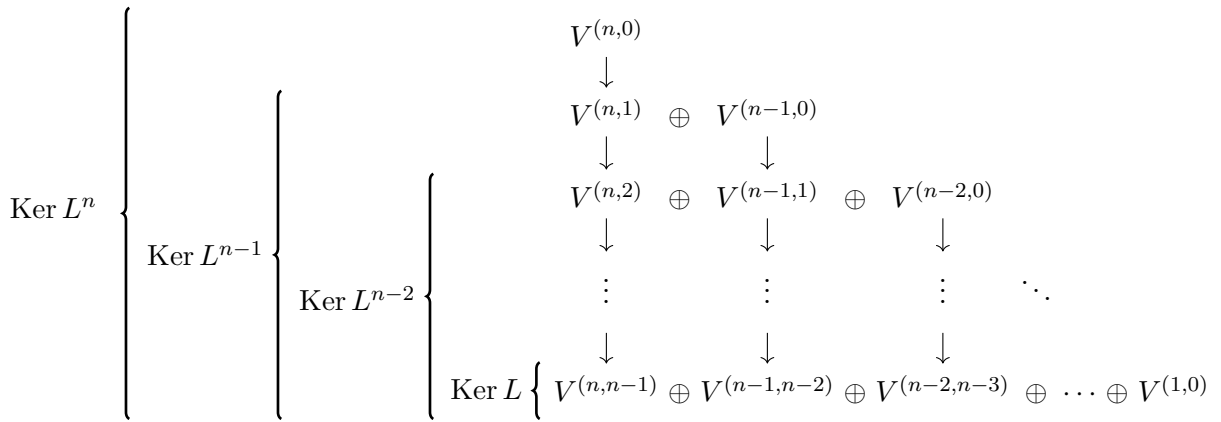
\begin{figure}[h]
\[
\begin{tikzpicture}[baseline={(0,0)}]
  \draw[decorate,decoration={brace},line width=1pt]  (-10, -2.75) node[shift={(-0.7cm,0.4cm)}] {$\Ker L$}  -- (-10, -2);
  \draw[decorate,decoration={brace},line width=1pt]  (-11.4, -2.75) node[shift={(-1cm,1.6cm)}] {$\Ker L^{n-2}$}  -- (-11.4, 0.5);
  \draw[decorate,decoration={brace},line width=1pt]  (-13.4, -2.75) node[shift={(-1cm,2.2cm)}] {$\Ker L^{n-1}$}  -- (-13.4, 1.6);
   \draw[decorate,decoration={brace},line width=1pt]  (-15.4, -2.75) node[shift={(-1cm,2.8cm)}] {$\Ker L^{n}$}  -- (-15.4, 2.7);
  \end{tikzpicture}
\begin{tikzcd}[row sep=1pc,column sep=-0.5pc]
 V^{(n,0)}\arrow{d}\\
 V^{(n,1)}\arrow{d}&\oplus&V^{(n-1,0)}\arrow{d}\\
 \tikzmark{bracebegin}V^{(n,2)}\arrow{d}&\oplus&V^{(n-1,1)}\arrow{d}&\oplus&V^{(n-2,0)}\arrow{d}\\
 \vdots\vphantom{\Bigl(} \arrow{d} &     & \vdots \vphantom{\Bigl(} \arrow{d}    &           &  \vdots\vphantom{\Bigl(}\arrow{d} &\ddots \\
 \tikzmark{braceend}V^{(n,n-1)}&\oplus&V^{(n-1,n-2)}&\oplus&V^{(n-2,n-3)}&\oplus&\cdots&\oplus& V^{(1,0)}
\end{tikzcd}
\]
\caption{Diagram of a Jordan decomposition}
\label{fig:Jordan}
\end{figure}
\begin{teo}[normal form for irreducible tensors]\label{teo: epsilon}
Let $T\in V_1\otimes \cdots\otimes V_e$ be a concise tensor, and
\[
    (L_1,\dots,L_e)\in\End(V_1)\times\cdots\times\End(V_e).
\]
For a positive natural number $n\geq 1$, the following conditions are equivalent:
\begin{enumerate}
    \item\label{it:epsilon:subalgebra} $(L_1,\dots,L_e)$ belongs to $\cen T$ and the subalgebra of $\cen T$
    generated by $(L_1,\dots,L_e)$ is isomorphic, as a $\kk$-algebra, to $\kk[\eps]/(\eps^n)$;
\item\label{it:epsilon:localForm} there exist tensors $T_1,\dots,T_n$ with
$$T_k\in\bigotimes_{i=1}^eL_i^{k-1}(\Ker L_i^k)$$
such that, for any $E_i$ such that $V_i=E_i\oplus\Ker L_i$, $T$ can be written as
\begin{equation}\label{formula:eqM}
T=\sum_{k=1}^n\sum_{\delta_1+\cdots+\delta_e=k-1}M_1^{\delta_1}\circ_1\cdots\circ_{e-1}M_e^{\delta_e}\circ_e T_{k},
\end{equation}
where $M_i$ is the inverse of the map $L_i\colon E_i \to {\im L_i}$.
 \item\label{it:epsilon:bases} for any $i=1,\dots,e$, there exists a decomposition
 $$V_i=\bigoplus_{\substack{q=1,\dots,n\\r=0,\dots,q-1}}V_i^{(q,r)},$$
with $\dim V_i^{(q,r)}=t_{i}^{(q)}$, such that, for any collection of bases $\bfx_i^{(q,r)}$ of $V_{i}^{(q,r)}$ such that $L(\bfx_i^{(q,r)})=\bfx_i^{(q,r+1)}$ for any  $r=0,\dots,q-2$ and $L(\bfx_i^{(q,q-1)})=0$, there exist $n$ tensors $T_{1},\dots,T_{n}$ with
$$T_{k}\in\bigotimes_{i=1,\dots,e}\bigoplus_{q=k}^n {\langle\bfx_i^{(q,q-1)}\rangle}$$
such that
    \begin{equation}\label{formula:Tk_explicit}
    T=\sum_{\substack{1\leq k\leq n\\ \delta_1+\cdots+\delta_e=k-1}}\sum_{\substack{k\leq q_i\leq n\\ 1\leq s_i\leq t_{i}^{(q_i)}}}T_{k}\biggl(\bigotimes_{i=1}^e\alpha_{i,s_i}^{(q_i,q_i-1)}\biggr)\bigotimes_{i=1}^e x_{i,s_i}^{(q_i,q_i-\delta_i-1)}
    \end{equation}
    where $\bfalpha_i^{(q,r)}$ is the dual basis of $\bfx_i^{(q,r)}$.
\end{enumerate}
\end{teo}
\begin{proof}
$\eqref{it:epsilon:localForm}\Rightarrow\eqref{it:epsilon:subalgebra}$. We
start by showing that $(L_1,\dots,L_e)\in\cen T$. For every $i=1, \ldots ,e$,
the composition $L_i\circ_i M_i \circ_i (-)$ is the
identity mapping on $\im L_i$. For clarity of notation, let us take $i=1$. We have
\begin{align*}
L_1\circ_1 T&=\sum_{k=1}^n\sum_{\delta_1+\cdots+\delta_e=k-1}L_1\circ_1 (M_1^{\delta_1})\circ_1 M_2^{\delta_2}\circ_2 \cdots\circ_{e-1}M_e^{\delta_e}\circ_e T_k\\
& = \sum_{k=1}^n\sum_{\delta_2+\cdots+\delta_e=k-1}M_2^{\delta_2}\circ_2 \cdots\circ_{e-1}M_e^{\delta_e}\circ_e(\underbrace{L_1\circ_1T_{k}}_{=0}))\\
&+\sum_{k=1}^n\sum_{\substack{\delta_1\geq 1\\\delta_1+\cdots+\delta_e=k-1}}M_1^{\delta_1-1}\circ_1 M_2^{\delta_2}\circ_2\cdots\circ_{e-1}M_e^{\delta_e}\circ_e T_k\\
& =\sum_{k=2}^n\sum_{\delta_1+\cdots+\delta_e=k-2}M_1^{\delta_1}\circ_1 \cdots\circ_{e-1}M_e^{\delta_e}\circ_e T_k.
\end{align*}
The final result does not distinguish the index $1$ in any way, so we obtain
the same expression starting from $L_2\circ_2 T$ etc., so
$L_1\circ_1 T =  \cdots = L_e\circ_e T$.
Hence, $(L_1,\dots,L_e)$ lies in $\cen T$.

Since the operators $L_i^n$ are zero, we have $(L_1, \ldots ,L_n)^n =0$. To
prove that $L_1^{n-1}\circ_1 T$ is nonzero, we argue as in the displayed
equations
above, this time acting with $L_1^{n-1}$, and get
\begin{equation}\label{eq:section}
    L_1^{n-1}\circ_1 T = \sum_{\delta_1=n-1,\ \delta_2 =
    \cdots = \delta_e = 0} (L_1^{n-1}\circ_1 M_1^{n-1})\circ_1 T_{n} = T_n,
\end{equation}
The $(V_1^{(n,0)}\otimes V_2\otimes\cdots\otimes V_e)$-component of $T$
is $M_1^{n-1}\circ_1 T_n$ and $T$ is concise, so $T_n\neq 0$. This shows
that $(L_1, \ldots ,L_{e})^{n-1}$ is a nonzero element of the centralizer and
so the subalgebra of $\cen{T}$ generated by $(L_1, \ldots ,L_e)$ is isomorphic
to $\kk[\eps]/(\eps^n)$.

%Following the same argument, it is easy to see that $L_i^n\circ_i T=0$ for any $i=1,\dots,d$, and thus, by \cite{JLP24}*{Lemma 3.1}, we get $L_i^n=0$. It remains to prove that $L_i^{n-1}\neq 0$. Consider the direct sum decomposition of $V_1\otimes\cdots\otimes V_e$ given by
%$$V_1\otimes\cdots\otimes V_e=\bigoplus_{\substack{q=1,\dots,n\\ r=0,\dots,q-1}}V_1^{q,r}\otimes V_2\otimes\cdots V_e.$$
%Since the $(V_1^{(n,0)}\otimes V_2\otimes\cdots\otimes V_e)$-component of $T$ is $M_1^{n-1}\circ_1 T_n$ and $T$ is concise, then $T_n\neq 0$.
%Finally, we have
%\[
%L_i^{n-1}\circ_i T=L_1^{n-1}\circ_1 T=\sum_{k=1}^n\sum_{\delta_1+\cdots+\delta_e=n-1}(L_1^{n-1}\circ M_1^{\delta_1})\circ_1 M_2^{\delta_2}\circ_2\cdots \circ_{e-1} M_e^{\delta_e}\circ_e T_k=T_n\neq 0.
%\]
%$(2)\Rightarrow(1)$ is trivial, while for $(1)\Rightarrow(2)$ we have to prove that $L_i^{n-1}\neq 0$ for every $i=1,\dots d$. Suppose by contradiction that there exists $i$ such that $L_i^{n-1}=0$. It follows that
%$$L_1^{n-1}\circ_1 T=\cdots =L_e^{n-1}\circ_e T=0,$$
%and thus, by \cite{JLP24}*{Lemma 3.1}, we have $(L_1,\dots,L_e)^{n-1}=(0,0,0)$, a contradiction.   It remains to prove $(2)\Leftrightarrow(3)$ and $(3)\Leftrightarrow(4)$.\\
    $\eqref{it:epsilon:subalgebra}\Rightarrow\eqref{it:epsilon:localForm}$.
    We use induction with respect to $n$. The base case, $\eps = 1$,
    implies that $L_1 = 0$, $L_2 = 0$, \ldots ,$L_e = 0$ and the implication
    is trivial.

    For the induction step, take a tensor $T$ as
    in~\eqref{it:epsilon:subalgebra} and let
    \begin{equation}\label{eq:sectionPrim}
        T'_n := L_1^{n-1} \circ_1 T = L_2^{n-1}\circ_2 T = \cdots =
        L_e^{n-1} \circ_e T.
    \end{equation}
    It follows that $T'_n \in L_1^{n-1}(V_1) \otimes L_2^{n-1}(V_2) \otimes
    \cdots \otimes L_e^{n-1}(V_e)$ and we can thus take the tensor
    \[
        T' =
        \sum_{\delta_1+\cdots+\delta_e=n-1}M_1^{\delta_1}\circ_1\cdots\circ_{e-1}M_e^{\delta_e}\circ_e
        T'_{n}.
    \]

    Using the implication~$\eqref{it:epsilon:localForm}\Rightarrow
    \eqref{it:epsilon:subalgebra}$ we learn that $(L_1, \ldots ,L_e)$ is in
    $\cen{T}$. Being compatible is a linear condition on the tensor, hence
    $(L_1, \ldots ,L_e)$ lies also in the centraliser of $T - T'$.
    By~\eqref{eq:section} and~\eqref{eq:sectionPrim}, we learn that
    $L_1^{n-1}\circ_1 T = T_n' = L_1^{n-1}\circ T'$. This implies that
    $(L_1^{n-1}, \ldots ,L_{e}^{n-1})$ annihilates $T - T'$ and we can apply
    the induction. By induction, the
    tensor $T - T'$ has the form~\eqref{formula:eqM}. Also $T'$ has this form,
    by construction, hence $T = (T - T') + T'$ also has the required form.

$\eqref{it:epsilon:localForm}\Leftrightarrow\eqref{it:epsilon:bases}$.
We start by proving the left to right implication.
    For any $i=1,\dots,e$, since $L_i$ is a nilpotent endomorphism with nilpotency index $n$, by \autoref{rem:Jordan} $V_i$ can be decomposed as
$$V_i=\bigoplus_{\substack{q=1,\dots,n\\r=0,\dots,q-1}}V_i^{(q,r)},$$
with $V_i^{(q,r)}=L^r(V_i^{(q,0)})$ and
$$\Ker L_i^j=V_i^{(n,n-j)}\oplus V_i^{(n-1,n-j-1)}\oplus\dots\oplus V_i^{(j,0)}\oplus\Ker L_i^{j-1}.$$
Let $\bfx_i^{(q,r)}$ a basis of $V_i^{(q,r)}$ such that $(x_i^{(q,r)})_{q,r}$ is a Jordan basis of $V_i$. By construction, we get
$$L_i^{k-1}(\Ker L_i^k)=\bigoplus_{q=k}^n{\langle\bfx_i^{(q,q-1)}\rangle}$$
and, for any $k,\delta_1,\dots,\delta_e\in\bbN$ with $1\leq k\leq n$ and $\delta_1+\dots+\delta_e=k-1$, we have
$$M_1^{\delta_1}\circ_1\cdots\circ_{e-1}M_e^{\delta_e}\circ_e \bigotimes_{i=1}^ex_{i,s_i}^{(q_i,q_i-1)}=\bigotimes_{i=1}^ex_{i,s_i}^{(q_i,q_i-\delta_i-1)}$$
for any $q_i=k,\dots,n$ and $s_i=1,\dots,t_i^{(q_i)}$.
From this equality, we get the equality of \autoref{formula:Tk_explicit}. To obtain the other implication, it is enough to reverse the argument.
\end{proof}
The following corollary, is an immediate consequence of \autoref{lemma:simmetriciCompatibili} and \autoref{teo: epsilon}. Before stating it, we need to introduce some notation.
\begin{cor}\label{cor: simmetrici epsilon}
Let $F\in S^d V$ be a concise homogeneous polynomial of degree $d$. The
following are equivalent:
\begin{enumerate}
    \item\label{it:symmetric:subalgebra} there exists a subalgebra of $\cen{F}$ which isomorphic to
        $\kk[\eps]/(\eps^n)$,
    \item\label{it:symmetric:form} there is a decomposition $V = \bigoplus_{1\leq q\leq n,\, 0\leq
            r\leq q-1}V^{(q, r)}$ of vector spaces and fixed isomorphisms
            {$$V^{(q, q-1)}\to V^{(q,q-2)} \to \cdots \to V^{(q, 0)},$$} such that
            there exist
            homogeneous polynomials
            $F_1, \ldots ,F_n$ with { $$F_k\in S^d \biggl(\bigoplus_{k\leq q\leq
            n} V^{(q,q-1)}\biggr)$$} such that
            \begin{equation}\label{eq:summationBody}
                F = \sum_{k=1}^n \sum_{\nu_1 + 2\nu_2 + \cdots + (n-1)\nu_{n-1} = k-1}
                \frac{1}{\nu_1! \nu_2! \cdots \nu_{n-1}!}D_1^{\nu_1}
                D_2^{\nu_2}
                \cdots D_{n-1}^{\nu_{n-1}} \hook F_k,
            \end{equation}
            where $D_i$ is the differential operator that is induced by the
            map $V^{(q, q-1)}\to V^{(q, q-1-i)}$, see~\eqref{eq:opAction}.
    \item\label{it:symmetric:coordinates} there is a basis $x^{(q, r)}_{i}$ of $V$ indexed by $1\leq r < q\leq
            n$ and $i=1, \ldots , t^{(q)}$ and homogeneous polynomials
            $F_1, \ldots ,F_n$ of degree $d$ such that for every $1\leq k\leq
            n$ we have
            \[
                F_{k}\in\kk[\bfx^{(q,q-1)}\;|\; q=k,\dots,n],
            \]
            where $\bfx^{(q, r)} := \left(
            x_1^{(q,r)},\dots,x_{t^{(q)}}^{(q,r)} \right)$, such that
            \[
                F=\sum_{k=1}^n \sum_{\nu_1 + 2\nu_2 \cdots + (n-1)\nu_{n-1} =
            k-1}
            \left( \sum_{q=k}^n \sum_{i=1}^{t^{(q)}} x_i^{(q, q-2)}\frac{\partial}{\partial
                x_{i}^{(q, q-1)}} \right)^{\nu_1}\hook \cdots \hook
%            \left( \sum_{q=k}^n \sum_{i=1}^{t^{(q)}} x_i^{q, q-3}\frac{\partial}{\partial
%                x_{i}^{(q, q-1)}} \right)^{\nu_2} \hook\ldots\hook
            \left( \sum_{q=k}^n \sum_{i=1}^{t^{(q)}} x_i^{(q, q-1-n)}\frac{\partial}{\partial
                x_{i}^{(q, q-1)}} \right)^{\nu_n}\hook F_k.
            \]
\end{enumerate}

%Let $F\in\bbC[x_1,\dots,x_m]_d$ a homogeneous concise polynomial of degree $d$. The following are equivalent:
%\begin{enumerate}
%\item there exists a subalgebra of $\cen{\calT(F)}$ which isomorphic, as a $\bbC$-algebra, to $\bbC[\eps]/(\eps^n)$;
%\item there exist a basis $(\bfx^{(q,r)})_{q,r}$ of $\bbC[x_1,\dots,x_m]$, with $\bfx^{(q,r)}=(x_1^{(q,r)},\dots,x_{t^{(q)}}^{(q,r)})$, for every $q=1,\dots,n$ and $r=0,\dots,q-1$, and $n$ degree $d$ forms $F_{1},\dots,F_{n}$, with
%\[
%F_{k}\in\bbC[\bfx^{(q,q-1)}\;|\; q=k,\dots,n]
%\]
%    for every $k=1,\dots,n$,
%    such that
%    $$F=\sum_{\substack{1\leq k\leq n\\ \delta_1+\cdots+\delta_d=k-1\\ \delta_1\leq\cdots\leq \delta_d}}\sum_{\substack{k\leq q_i\leq n\\ 1\leq s_i\leq t^{(q_i)}}}\frac{\partial^dF_{k}}{\partial x^{(q_1,q_1-1)}_{s_1}\cdots\partial x^{(q_d,q_d-1)}_{s_d}}\prod_{i=1}^dx^{(q_i,q_i-\delta_i-1)}_{s_i}.$$
%\end{enumerate}
\end{cor}
\begin{proof}
    To prove
    $\eqref{it:symmetric:subalgebra}\Rightarrow\eqref{it:symmetric:form}$ we
    use \autoref{teo: epsilon} and its notation, in particular $(L_1, \ldots
    ,L_d)$ and $(M_1, \ldots ,M_d)$. We view $F$ as a symmetric tensor
    by~\eqref{eq:symmetric}. Using
    \autoref{lemma:simmetriciCompatibili} we learn that $L_1 =  \cdots = L_d =
    L$ and $M_1 =  \cdots = M_d$.

    By \autoref{teo: epsilon}, $F$ can be written as in
    \autoref{formula:Tk_explicit}. The tensor $T_n$ appearing there is
    obtained by action of $L^{n-1}$ on $F$, hence is symmetric. The other
    tensors $T_{n-1}, \ldots ,T_1$ are obtained similarly from
    \emph{polynomials} obtained by subtracting polynomials from $F$, hence
    $T_{n-1}, \ldots ,T_1$ are symmetric as well. We
    denote the corresponding polynomials as $F_1, \ldots ,F_n$.

    In remains to identify the action~\eqref{formula:eqM}
    with~\eqref{eq:summationBody} for every fixed $1\leq k\leq n$. By our global assumption, the
    characteristic of $\kk$ is zero or greater than $d$. Hence, the space $S^d
    \bigoplus_{i\leq q\leq n} V^{(q,q-1)}$ is spanned by $\ell^{d}$ where
    $\ell$ ranges over the elements of $\bigoplus_{i\leq q\leq n}
    V^{(q,q-1)}$.

    The polynomial $\ell^d$ identifies with the symmetric tensor $\ell
    \ootimes \ell$. The summation
    \[
        \sum_{\delta_1+\cdots+\delta_d=k-1}M^{\delta_1}\circ_1\cdots\circ_{d-1}M^{\delta_d}\circ_d
        \ell^{\otimes d}
    \]
    yields $\sum_{\delta_1+\cdots+\delta_d=k-1} M^{\delta_1}(\ell) \cdot
    \cdots \cdot M^{\delta_d}(\ell)$, since the right-hand-side is the
    symmetrization of the result and the result is already symmetric. The
    summation still remembers the order of $\delta_{\bullet}$ even though the
    result does not depend on it. Now we rearrange the summation: we forget about the order and instead
    group the possible $\delta_{\bullet}$ according to the sequence
    $(\nu_1,\nu_2, \ldots ,\nu_{n-1})$, where $\nu_i = |\left\{j
    \ |\ \delta_j = i\right\}|$ remembers how many times
    $i$ appears in $\delta_{\bullet}$. For a given $(\nu_1, \ldots
    ,\nu_{n-1})$, the number of possible $\delta_{\bullet}$ is given by a
    multinomial coefficient and so we obtain
    \[
        \sum_{\nu_1 + 2\nu_2 +  \cdots + (n-1)\nu_{n-1} = k-1}
            \frac{d!}{\nu_1! \cdots \nu_d!(d-\sum_{i=1}^{n-1}
    \nu_i)!}(M^{1}(\ell))^{\nu_1} \cdots
    (M^{n-1}(\ell))^{\nu_{n-1}}
    \]
    The differential operator $D_1^{\nu_1} \cdots
    D_{n-1}^{\nu_{n-1}}$ applied to $\ell^d$ yields
    \[
        d(d-1) \cdots \biggl(d+1-\sum_{i=1}^{n-1} \nu_i\biggr) \bigl(M^{1}(\ell)\bigr)^{\nu_1} \cdots
        \bigl(M^{n-1}(\ell)\bigr)^{\nu_{n-1}} = \frac{d!}{(d-\sum_{i=1}^{n-1}
    \nu_i)!}\bigl(M^{1}(\ell)\bigr)^{\nu_1}  \cdots
        \bigl(M^{n-1}(\ell)\bigr)^{\nu_{n-1}}
    \]
    Comparing the two displayed equations, we obtain the
    desired~\eqref{eq:summationBody}. Reserving the argument, we obtain
    $\eqref{it:symmetric:form}\implies\eqref{it:symmetric:subalgebra}$. To
    prove~$\eqref{it:symmetric:form}\Leftrightarrow\eqref{it:symmetric:coordinates}$,
    we just express the operators $D_i$ in coordinates.
\end{proof}

\begin{example}
    Let us make a very concrete example. Fix $n = 3$, $V = V^{(3, 2)} \oplus
    V^{(3, 1)} \oplus V^{(3, 0)}$ with operator $M$ yielding isomorphisms
    $V^{(3, 2)}\to V^{(3, 1)}$ and $V^{(3, 1)}\to V^{(3, 0)}$. Additionally,
    assume that $V^{(3, 2)}$ is spanned by $x^{(3, 2)}$. Take $F_2 =
    (x^{(3, 2)})^3$. The expression in \autoref{cor: simmetrici
    epsilon} yields
    \[
        3 (x^{(3, 2)})^2 x^{(3, 0)} + \frac{1}{2}\cdot 6 x^{(3, 2)}(
        x^{(3, 1)}
        )^2.
    \]
\end{example}

\section{Limits of direct sums}

In this section we deduce that all tensors with a subalgebra
$\kk[\eps]/(\eps^n)$ in the centroid are limits of direct sums with $n$
factors. This result is quite unexpected, since it yields an easy way of
constructing interesting such limits.

\begin{prop}\label{prop:direct_sum_limit}
    Let $T\in S^{d_1}V_1\otimes\cdots\otimes S^{d_e}V_e$ be a concise tensor. If there exists a
subalgebra of $\cen{T}$ isomorphic to $\kk[\eps]/(\eps^n)$, then $T$ is a
limit of direct sums of the form $T^{(1)} +  \cdots + T^{(n)}$.
\end{prop}

\begin{proof}
    Let us first consider the \textbf{Segre format}, so we assume $d_1 =
    \cdots = d_e = 1$. The general case will follow easily from this subcase.

Let $(L_1,\dots,L_e)\in\cen T$ generate the subalgebra of $\cen T$ isomorphic
to $\kk[\eps]/(\eps^n)$.
Then, by \Cref{teo: epsilon}, there exist $T_1,\dots,T_n$ with
$$T_k\in\bigotimes_{i=1}^e L_i^{k-1}(\Ker L_i^k)$$
such that, for any $E_i$ such that $V_i=E_i\oplus\Ker L_i$, $T$ can be written as
\begin{equation}\label{eq:desiredTensor}
T=\sum_{k=1}^n\sum_{\delta_1+\cdots+\delta_e=k-1}M_1^{\delta_1}\circ_1\cdots\circ_{d-1}M_e^{\delta_e}\circ_e T_{k},
\end{equation}
where $M_i$ is the inverse of the map $L_i\colon {E_i}\to {\im L_i}$.

The field $\kk$ is algebraically closed, hence infinite. Fix pairwise distinct
elements $\omega_1, \ldots ,\omega_n$ of $\kk$. (The use of the letter
$\omega$ will be explained in \autoref{ref:rootsOfUnitChoice:ex}). Since these elements are distinct,
for every $1\leq k\leq n$, the vectors
\begin{align*}
   & (1, 1, 1, \ldots ,1)\\
   & (\omega_1, \omega_2, \ldots, \omega_k)\\
   & (\omega_1^2, \omega_2^2, \ldots, \omega_k^2)\\
   & \phantom{mmm} \vdots \\
   &  (\omega_1^{k-1}, \omega_2^{k-1}, \ldots, \omega_k^{k-1})
\end{align*}
are linearly independent, by the Vandermonde's determinant, so there exist coefficients $\alpha_{k,1}, \ldots
,\alpha_{k,k}$ in $\kk$ such that
\begin{align}\label{eq:coeffs}
    & \alpha_{k,1}\omega_1^{\gamma-1} +  \cdots +
    \alpha_{k,k}\omega_k^{\gamma-1} =
    0\quad\mbox{for every}\quad 1\leq \gamma\leq k-1\\\notag
    & \alpha_{k,1}\omega_1^{k-1} +  \cdots + \alpha_{k,k}\omega_k^{k-1} = 1
\end{align}
\newcommand{\Mtw}{\widetilde{M}}%
\newcommand{\Vtw}{\widetilde{V}}%
\newcommand{\Mtwpar}[2]{\Mtw_{#1, #2}}%

Let us use the decomposition from \autoref{rem:Jordan} so that every $V_i$
decomposes as
\[
    V_i=\bigoplus_{\substack{q=1,\dots,n\\r=0,\dots,q-1}}V_i^{(q,r)},
\]
with $L_i^{k-1}(\Ker L_i) = V_i^{(n,n-1)}\oplus  \cdots \oplus V_i^{(k,k-1)}$
for every $1\leq k\leq n$.
For every $1\leq i\leq e$, $1\leq j\leq n$ define the
linear operator $\Mtwpar{i}{j}\colon L_i^{j-1}(\Ker L_i)\to V_i^{(\geq j,
\bullet)}[t]$, depending on the
formal variable $t$, as follows. For every $j\leq q \leq n$, we have
\begin{equation}\label{eq:Mdef}
    \Mtwpar{i}{j}(t)|_{V^{(q,q-1)}} := \Id_{V^{(q, q-1)}} + t \omega_{j} M_i + (t\omega_j M_i)^2 +
    \cdots + (t\omega_j M_i)^{q-1} \colon V_i^{(q,q-1)}\to
    V_i^{(q, \bullet)}[t]
\end{equation}
All these restrictions are well defined. Observe that
\begin{equation}\label{eq:residue}
    \Mtwpar{i}{j}(t) \equiv \Id + t \omega_{j} M_i + (t\omega_j M_i)^2 +
    \cdots + (t\omega_j M_i)^{j-1}\colon L_i^{j-1}(\Ker L_i)\to V_i^{(\geq
    j,\bullet)}[t] \mod t^j
\end{equation}
Consider the family of tensors, parameterised by $t$, given by
\[
    \sum_{k=1}^n t^{n-k}\sum_{j=1}^k \alpha_{k,j}\cdot
    \Mtwpar{1}{j}(t)\circ_1
    \Mtwpar{2}{j}(t)\circ_2 \cdots \Mtwpar{e}{j}(t)\circ_e T_k.
\]
We compute that
\begin{align*}
    &\frac{1}{t^{n-1}}\sum_{k=1}^n t^{n-k}\sum_{j=1}^k \alpha_{k,j}\cdot
    \Mtwpar{1}{j}(t)\circ_1
    \Mtwpar{2}{j}(t)\circ_2 \cdots \Mtwpar{e}{j}(t)\circ_e T_k
    \stackrel{\eqref{eq:residue}}{=}\\
    &\frac{1}{t^{n-1}}\sum_{k=1}^n t^{n-k}\sum_{j=1}^k
    \alpha_{k,j}\left(\sum_{\gamma=1}^{k}\sum_{\delta_1+\cdots+\delta_e=\gamma-1}
    (t\omega_j M_1)^{\delta_1} \circ_1
    (t\omega_j M_2)^{\delta_2} \circ_2 \cdots
    (t\omega_j M_e)^{\delta_e} \circ_e T_k + t^{k}\left(  \ldots
    \right)\right) = \\
    &\frac{1}{t^{n-1}}\sum_{k=1}^n t^{n-k}\left(\sum_{\gamma=1}^{k}\left( \sum_{j=1}^k
    \alpha_{k,j}\omega_j^{\gamma-1} \right)t^{\gamma-1}
    \sum_{\delta_1+\cdots+\delta_e=\gamma-1}M_1^{\delta_1}\circ_1\cdots
    \circ_{e-1}M_e^{\delta_e}\circ_e T_k + t^{k}( \ldots )\right) \stackrel{\eqref{eq:coeffs}}{=}\\
    &\frac{1}{t^{n-1}}\sum_{k=1}^n t^{n-1} \left( \sum_{\delta_1+\cdots+\delta_e=k-1}M_1^{\delta_1}\circ_1\cdots
    \circ_{e-1}M_e^{\delta_e}\circ_e T_k + t( \ldots )\right)
    \stackrel{\eqref{eq:desiredTensor}}{=} T + t( \ldots ),
\end{align*}
where the part $t^k( \ldots )$ appears since $T_k$ may have nonzero parts in
$V^{(q,q-1)}$ for $q > k$ and then $\Mtwpar{\bullet}{j}$ contain summands of the
form $t^{\ell} \omega_{j}^{\ell} M_{\bullet}^{\ell}$ for $\ell \geq k$, which
are not taken into account in the previous sum.

The displayed equation proves that $T$ is a limit of the tensors of the form
\[
    \sum_{k=1}^n \sum_{j=1}^k t^{n-k}\alpha_{k,j}\cdot
    \Mtwpar{1}{j}\circ_1
    \Mtwpar{2}{j}\circ_2 \cdots \Mtwpar{e}{j}\circ_e T_k =
    \sum_{j=1}^n \left(\sum_{k=j}^n t^{n-k}\alpha_{k,j}\cdot
    \Mtwpar{1}{j}\circ_1
    \Mtwpar{2}{j}\circ_2 \cdots \Mtwpar{e}{j}\circ_e T_k\right).
\]
It remains to deduce that these, for $t\neq 0$, yield the desired direct sum
with $n$ summands.

Fix a coordinate $1\leq i\leq e$, take a nonzero $\lambda\in \kk$ and consider the map
\[
    \Mtw_i(\lambda)\colon V_i^{(n,n-1)} \oplus \left( V_i^{(n,n-1)}\oplus V_i^{(n-1, n-2)} \right) \oplus
    \cdots \oplus \left( V^{(n,n-1)} \oplus V_i^{(n-1, n-2)} \oplus \cdots
    \oplus V_i^{(1, 0)} \right)\to V_i
\]
given by $\Mtw_i(\lambda) := \Mtwpar{i}{n}(\lambda) \oplus \Mtwpar{i}{n-1}(\lambda) \oplus \cdots
\oplus \Mtwpar{i}{1}(\lambda)$. We claim that this map is an isomorphism. Both
sides have same dimension
\[
    n\dim_{\kk} V_i^{(n,n-1)} + (n-1)\dim_{\kk} V_i^{(n-1,n-2)}+ \cdots + \dim_{\kk} V_i^{(1, 0)},
\]
so that it is enough to show
that the map is surjective. To do this, it is enough to show that the image
contains $V^{(k,k-1)}\oplus  \cdots \oplus V^{(k,0)}$ for every $1\leq k\leq n$.

Restrict $\Mtw_i(\lambda)$ to the subspace $\underbrace{V^{(k,k-1)}\oplus  \cdots \oplus
V^{(k,k-1)}}_{k}$, which appears in the domain. The image of this subspace is
in $V^{(k,k-1)}\oplus  \cdots \oplus V^{(k,0)}$ and, by~\eqref{eq:Mdef}, the restriction is given by a
diagonal matrix with blocks
\[
    \begin{pmatrix}
        1 & 1 &  \ldots & 1\\
        \lambda\omega_1 & \lambda\omega_2 & \ldots & \lambda\omega_k\\
        (\lambda\omega_1)^2 & (\lambda\omega_2)^2 & \ldots & (\lambda\omega_k)^2\\
         &&\ldots\\
         (\lambda\omega_1)^{k-1} & (\lambda\omega_2)^{k-1} & \ldots &
         (\lambda\omega_k)^{k-1}\\
    \end{pmatrix}
\]
hence the map $\Mtw_i(\lambda)$ is indeed an isomorphism.
For every $1\leq i\leq e$ and $1\leq k\leq n$ define
\[
    \Vtw_i^{(k)}(\lambda) := \Mtw_i(\lambda)\left( V_i^{(n,n-1)}\oplus  \cdots
    \oplus
    V_i^{(k,k-1)} \right) = \Mtwpar{i}{k}(\lambda)\left( V_i^{(n,n-1)}\oplus  \cdots
    \oplus
    V_i^{(k,k-1)} \right),
\]
so that $V_i = \Vtw_i^{(1)}(\lambda)\oplus \cdots \oplus
\Vtw_i^{(n)}(\lambda)$. Directly by definitions, { $$T^{(j)}(\lambda) \coloneqq \sum_{k=j}^n
\lambda^{n-k}\alpha_{k,j}\Mtwpar{1}{j}(\lambda)\circ_1
\Mtwpar{2}{j}(\lambda)\circ_2 \cdots \Mtwpar{e}{j}(\lambda)\circ_e T_k$$} lies in
$V_1^{(j)}(\lambda)\ootimes V_e^{(j)}(\lambda)$ so
\begin{align*}
    &\frac{1}{\lambda^{n-1}}\sum_{k=1}^n \sum_{j=1}^k \lambda^{n-k}\alpha_{k,j}\cdot
    \Mtwpar{1}{j}(\lambda)\circ_1
    \Mtwpar{2}{j}(\lambda)\circ_2 \cdots \circ_{e-1}\Mtwpar{e}{j}(\lambda)\circ_e T_k =\\
    &\frac{1}{\lambda^{n-1}}\sum_{j=1}^n\sum_{k=j}^n \lambda^{n-k}\alpha_{k,j}\cdot
    \Mtwpar{1}{j}(\lambda)\circ_1
    \Mtwpar{2}{j}(\lambda)\circ_2 \cdots \circ_{e-1}\Mtwpar{e}{j}(\lambda)\circ_e T_k = \frac{1}{\lambda^{n-1}}\sum_{j=1}^n
    T^{(j)}(\lambda)
\end{align*}
is a direct sum with $n$ nonzero factors.

Now, let us go back to \textbf{arbitrary format}, that is,
{
\[
    T\in S^{d_1}V_1 \ootimes S^{d_e}V_e\injects\underbrace{V_1\ot
    \cdots \ot V_1}_{d_1} \ot \underbrace{V_2 \ootimes V_2}_{d_2}\ootimes \underbrace{V_e \ootimes
    V_e}_{d_e}.
\]
}
From the above, we know that $T$ is a limit of direct sums $T^{(1)} +  \cdots
+ T^{(n)}$, where $T^{(j)}$ lie in $V_1^{\otimes d_1}\ootimes V_e^{\otimes d_e}$. Thanks to \autoref{lemma:simmetriciCompatibili}, we see that
actually $T^{(1)}, \ldots ,T^{(n)}$ lie in the subspace $S^{d_1}V_1 \ootimes
S^{d_e}V_e$, which yields the desired claim for $T$ in the Segre-Veronese
format.
\end{proof}

\begin{example}\label{ref:rootsOfUnitChoice:ex}
    Let us now discuss why $\omega_1, \ldots ,\omega_n$ are denoted so in the
    proof of \autoref{prop:direct_sum_limit}.
    In the setup of \autoref{rem:Jordan}, suppose that there is only one
    column, so that $V = V^{(n,\bullet)}$. In this case, take $\omega_1,
    \ldots ,\omega_n$ to be $n$-th roots of unity in $\kk$. The advantage of
    this choice is that
    \[
        \sum_{j=1}^n\omega_{j}^{\gamma}=0
    \]
    for every $\gamma=1,\dots,n-1$ and
    \[
        \sum_{j=1}^n\omega_{j}^{n}=n,
    \]
    so the coefficients $\alpha_{n, \bullet}$ become very easy. Regretfully, there
    seems to be no such nice choice for more than one column, in general.
\end{example}

%As a direct consequence of \Cref{prop:direct_sum_limit}, we get the equivalent statement for the symmetric case.
%\begin{cor}
%Let $F\in \bbC[x_1,\dots,x_m]$ be a concise tensor. If there exists a subalgebra of $\cen{\calT(F)}$ isomorphic to $\bbC[\eps]/(\eps^n)$, then $F$ is a limit of direct sums.
%\end{cor}

\begin{exam}
Let $V=\kk^6$ and
$\calB=(x^{(3,2)},x^{(3,1)},x^{(3,0)},x^{(2,1)},x^{(2,0)},x^{(1,0)})$ a basis
for $V$. Consider the cubic polynomial $T\in S^3 V$, defined as
\[
T= x^{(3,2)} x^{(2,1)} x^{(1,0)}-2x^{(3,1)} x^{(3,2)}
x^{(2,1)}-(x^{(3,2)})^2 x^{(2,0)}+3x^{(3,0)}
(x^{(3,2)})^2+3(x^{(3,1)})^2  x^{(3,2)}.
\]
A straightforward computation shows that $(L,L,L)\in \cen T$, where $L$ is the
endomorphism of $\kk^6$ whose matrix with respect to the basis $\calB$ is
$$
\begin{pmatrix}
0 & 1 & 0 & 0 & 0 & 0\\
0 & 0 & 1 & 0 & 0 & 0\\
0 & 0 & 0 & 0 & 0 & 0\\
0 & 0 & 0 & 0 & 1 & 0\\
0 & 0 & 0 & 0 & 0 & 0\\
0 & 0 & 0 & 0 & 0 & 0
\end{pmatrix}.
$$
The operator $L$ is nilpotent with nilpotency index $3$ and in our basis it
yields {$$
L(\an{x^{(3, 2)}, x^{(2,1)}, x^{(1,0)}}) = 0,\quad L(x^{(3, 0)}) =
x^{(3, 1)},\quad L(x^{(3,1)}) = x^{(3, 2)},\quad L(x^{(2,0)}) = x^{(2,1)}.$$} 
We know that, by
\autoref{cor: simmetrici epsilon}, it is possible to  write $T$ as
\[
    T = T_1 + x^{(2, 0)}\frac{\partial T_2}{\partial x^{(2,1)}} + x^{(3,
    1)}\frac{\partial T_2}{\partial x^{(3, 2)}} + x^{(3,
    0)}\frac{\partial T_3}{\partial x^{(3, 2)}}
    +\frac{1}{2}\left(x^{(3, 1)}\right)^2\frac{\partial^2 T_3}{\partial^2 x^{(3, 2)}}
\]
where $T_1 \in S^3(\an{x^{(3, 2)}, x^{(2,1)}, x^{(1,0)}})$, $T_2\in
S^3(\an{x^{(3, 2)}, x^{(2, 1)}})$, $T_3\in S^3(\an{x^{(3, 2)}})$.
Following the proof or directly from the form of $T$, we can compute
\[
    T_3=(x^{(3,2)})^3,\quad T_2=-x^{(2,1)} (x^{(3,2)})^2,\quad T_1=x^{(1,0)} x^{(2,1)} x^{(3,2)}.
\]
The inverse of the map \[
L\colon\langle x^{(3,0)},x^{(3,1)},x^{(2,0)}\rangle\to\langle x^{(3,1)},x^{(3,2)},x^{(2,1)}\rangle\] is the map \[M\colon \langle x^{(3,1)},x^{(3,2)},x^{(2,1)}\rangle\to \langle x^{(3,0)},x^{(3,1)},x^{(2,0)}\rangle,
\]
defined by the relations
\[
M(x^{(3,1)})=x^{(3,0)},\quad M(x^{(3,2)})=x^{(3,1)},\quad M(x^{(2,1)})=x^{(2,0)}.
\]
In order to satisfy the conditions on $\omega_1,\omega_2,\omega_3$ and on $\alpha_{k,j}$ for $1\leq k\leq 3$ and $1\leq j\leq k$, we can choose $\omega_1=1,\omega_2=0,\omega_3=-1$ and 
$$\alpha_{1,1}=1,\quad (\alpha_{2,1},\alpha_{2,2})=(1,-1),\quad (\alpha_{3,1},\alpha_{3,2},\alpha_{3,3})=(1/2,-1,1/2).$$
We consider now the operators $\tilde{M}_1,\tilde{M}_2$ and $\tilde{M}_3$ defined, according to formulas \eqref{eq:Mdef} and \eqref{eq:residue}, by the following restrictions
\begin{align*}
    \tilde{M}_{1}(t)|_{V^{(3,2)}} &\coloneqq \Id_{V^{(3, 2)}} + tM + t^2M^2\colon V^{(3,2)}\to
    (V^{(3,2)}\oplus V^{(3,1)}\oplus V^{(3,0)})[t]\\
    \tilde{M}_{2}(t)|_{V^{(3,2)}} &\coloneqq \Id_{V^{(3, 2)}} \colon V^{(3,2)}\to
    (V^{(3,2)}\oplus V^{(3,1)}\oplus V^{(3,0)})[t]\\    \tilde{M}_{3}(t)|_{V^{(3,2)}} &\coloneqq \Id_{V^{(3, 2)}} - t M + t^2 M^2\colon V^{(3,2)}\to
    (V^{(3,2)}\oplus V^{(3,1)}\oplus V^{(3,0)})[t]\\
    \tilde{M}_{1}(t)|_{V^{(2,1)}} &\coloneqq \Id_{V^{(2, 1)}} + tM\colon V^{(2,1)}\to
   (V^{(2,1)}\oplus V^{(2,0)})[t]\\
        \tilde{M}_{2}(t)|_{V^{(2,1)}} &\coloneqq \Id_{V^{(2, 1)}} \colon V^{(2,1)}\to
    (V^{(2,1)}\oplus V^{(2,0)})[t]\\
    \tilde{M}_{1}(t)|_{V^{(1,0)}} &\coloneqq \Id_{V^{(1, 0)}}\colon V^{(1,0)}\to
    V^{(1,0)}[t].
\end{align*}

%\begin{align*}
%    \tilde{M}_{1}(t)|_{V^{(3,2)}} &\coloneqq \Id_{V^{(3, 2)}} + t \omega_{1} %M + t^2\omega_1^2 M^2\colon V^{(3,2)}\to
%    (V^{(3,2)}\oplus V^{(3,1)}\oplus V^{(3,0)})[t]\\
%    \tilde{M}_{2}(t)|_{V^{(3,2)}} &\coloneqq \Id_{V^{(3, 2)}} + t \omega_{2} M + t^2\omega_2^2 M^2\colon V^{(3,2)}\to
%    (V^{(3,2)}\oplus V^{(3,1)}\oplus V^{(3,0)})[t]\\
%    \tilde{M}_{3}(t)|_{V^{(3,2)}} &\coloneqq \Id_{V^{(3, 2)}} + t \omega_{3} M + t^2\omega_3^2 M^2\colon V^{(3,2)}\to
%    (V^{(3,2)}\oplus V^{(3,1)}\oplus V^{(3,0)})[t]\\
%    \tilde{M}_{1}(t)|_{V^{(2,1)}} &\coloneqq \Id_{V^{(2, 1)}} + t \omega_{1} M\colon V^{(2,1)}\to
%    (V^{(2,1)}\oplus V^{(2,0)})[t]\\
%        \tilde{M}_{2}(t)|_{V^{(2,1)}} &\coloneqq \Id_{V^{(2, 1)}} + t \omega_{2} M\colon V^{(2,1)}\to
%    (V^{(2,1)}\oplus V^{(2,0)})[t]\\
%    \tilde{M}_{1}(t)|_{V^{(1,0)}} &\coloneqq \Id_{V^{(1, 0)}}\colon V^{(1,0)}\to
%    V^{(1,0)}[t].
%\end{align*}
We have
\begin{align}
\label{formula:example}
\tilde{M}_1(t)\circ T_1&=x^{(1,0)}\bigl(x^{(2,1)}+t x^{(2,0)}\bigr)(x^{(3,2)}+t x^{(3,1)}+t^2x^{(3,0)}),\\
\notag\tilde{M}_1(t)\circ T_2&=-(x^{(2,1)}+t x^{(2,0)})(x^{(3,2)}+t x^{(3,1)}+t^2 x^{(3,0)})^2,\\
\notag\tilde{M}_1(t)\circ T_3&=(x^{(3,2)}+t x^{(3,1)}+t^2 x^{(3,0)})^3,\\
\notag\tilde{M}_2(t)\circ T_2&=-x^{(2,1)}(x^{(3,2)})^2,\\
\notag\tilde{M}_2(t)\circ T_3&=(x^{(3,2)})^3,\\
\notag\tilde{M}_3(t)\circ T_3&=(x^{(3,2)}-t x^{(3,1)}+t^2 x^{(3,0)})^3.
\end{align}
%\begin{align}
%\label{formula:example}
%\tilde{M}_1(t)\circ T_1&=x^{(1,0)}\bigl(x^{(2,1)}+t\omega_1 x^{(2,0)}\bigr)(x^{(3,2)}+t\omega_1 x^{(3,1)}+t^2\omega_1^2x^{(3,0)}),\\
%\notag\tilde{M}_1(t)\circ T_2&=-(x^{(2,1)}+t\omega_1 x^{(2,0)})(x^{(3,2)}+t\omega_1 x^{(3,1)}+t^2\omega_1^2 x^{(3,0)})^2,\\
%\notag\tilde{M}_1(t)\circ T_3&=(x^{(3,2)}+t\omega_1 x^{(3,1)}+t^2\omega_1^2 x^{(3,0)})^3,\\
%\notag\tilde{M}_2(t)\circ T_2&=-(x^{(2,1)}+t\omega_2 x^{(2,0)})(x^{(3,2)}+t\omega_2 x^{(3,1)}+t^2\omega_2^2 x^{(3,0)})^2,\\
%\notag\tilde{M}_2(t)\circ T_3&=(x^{(3,2)}+t\omega_2 x^{(3,1)}+t^2\omega_2^2 x^{(3,0)})^3,\\
%\notag\tilde{M}_3(t)\circ T_3&=(x^{(3,2)}+t\omega_3 x^{(3,1)}+t^2\omega_3^2 x^{(3,0)})^3.
%\end{align}
By setting
$$
    S_t\coloneqq t^2\tilde{M}_1(t)\circ T_1+t(\tilde{M}_1(t)\circ T_2-\tilde{M}_2(t)\circ T_2)+\frac{1}{2}\tilde{M}_1(t)\circ T_3-\tilde{M}_2(t)\circ T_3+\frac{1}{2}\tilde{M}_3(t)\circ T_3,
$$
%\begin{align*}
%    S_t&\coloneqq t^2\alpha_{1,1}\tilde{M}_1(t)\circ T_1+t(\alpha_{2,1}\tilde{M}_1(t)\circ T_2+\alpha_{2,2}\tilde{M}_2(t)\circ T_2)+\alpha_{3,1}\tilde{M}_1(t)\circ T_3+\alpha_{3,2}\tilde{M}_2(t)\circ T_3\\
%    &\hphantom{{}\coloneqq{}}+ \alpha_{3,3}\tilde{M}_3(t)\circ T_3,
%\end{align*}
we have to consider know, according to \Cref{prop:direct_sum_limit}, the limit 
\[
\lim_{t\to 0}\frac{1}{t^2}S_t.
\]
By substituting \ref{formula:example} in $S_t$, we get
\begin{align*}
S_t&=t^2\Bigl(x^{(1,0)}x^{(2,1)}x^{(3,2)}-x^{(2,0)}(x^{(3,2)})^2-2x^{(2,0)}x^{(3,2)}x^{(3,1)}+3\bigl(x^{(3,2)}(x^{(3,1)})^2+(x^{(3,2)})^2x^{(3,0)}\bigr)\Bigr)\\
&\hphantom{{}={}}+t^3(\dots).
\end{align*}
%\begin{align*}
%S_t&=(\alpha_{3,1}+\alpha_{3,2}+\alpha_{3,3})(x^{(3,2)})^3\\
%&\hphantom{{}={}}+t(-(\alpha_{2,1}+\alpha_{2,2})x^{(2,1)}(x^{(3,2)})^2+3(\alpha_{3,1}\omega_1+\alpha_{3,2}\omega_2+\alpha_{3,3}\omega_3)(x^{(3,2)})^2x^{(3,1)})\\
%&\hphantom{{}={}} +t^2\Bigl(\alpha_{1,1}x^{(1,0)}x^{(2,1)}x^{(3,2)}-(\alpha_{2,1}\omega_1+\alpha_{2,2}\omega_2)x^{(2,0)}(x^{(3,2)})^2\\
%&\hphantom{{}={}}\qquad\ +3(\alpha_{3,1}\omega_1^2+\alpha_{3,2}\omega_2^2+\alpha_{3,3}\omega_3^2)\bigl(x^{(3,2)}(x^{(3,1)})^2+(x^{(3,2)})^2x^{(3,0)}\bigr)\Bigr)+t^3(\dots).
%\end{align*}
We get $S_t=t^2T+t^3(\dots)$, so that 
\[
\lim_{t\to 0}\frac{1}{t^2}S_t=T.
\]
To see that $T$ is indeed a limit of direct sums, it is enough to compute
$$T^{(j)}(t) \coloneqq \sum_{k=j}^n t^{n-k}\alpha_{k,j}\tilde{M}_{j}(t)\circ T_k$$
for $j=1,2,3$. We have
\begin{align*}
    T^{(1)}(t)&=t^2x^{(1,0)}\bigl(x^{(2,1)}+t x^{(2,0)}\bigr)(x^{(3,2)}+t x^{(3,1)}+t^2x^{(3,0)})\\
    &\hphantom{{}={}}-t(x^{(2,1)}+t x^{(2,0)})(x^{(3,2)}+t x^{(3,1)}+t^2 x^{(3,0)})^2+\frac{1}{2}(x^{(3,2)}+t x^{(3,1)}+t^2 x^{(3,0)})^3\\
    T^{(2)}(t)&=tx^{(2,1)}(x^{(3,2)})^2+\frac{1}{2}(x^{(3,2)})^3\\
    T^{(3)}(t)&=\frac{1}{2}(x^{(3,2)}-t x^{(3,1)}+t^2 x^{(3,0)})^3.
\end{align*}
Therefore, we have $S_t=T^{(1)}+T^{(2)}+T^{(3)}$ and thus $S_t$ is a direct sum for any $t$ because it belongs to
$$S^3\bigl(\langle x^{(1,0)},x^{(2,1)}+tx^{(2,0)},x^{(3,2)}+tx^{(3,1)}+t^2x^{(3,0)}\rangle\bigr)\oplus S^3\bigl(\langle x^{(3,2)},x^{(2,1)}\rangle\bigr)\oplus S^3\langle x^{(3,2)}-t x^{(3,1)}+t^2 x^{(3,0)}\rangle.$$
\end{exam}

\bibliographystyle{alpha}
\bibliography{111algebras.bib}

\end{document}